\documentclass[twoside,12pt,a4paper,leqno]{amsart}

\usepackage{amsmath,amsfonts,amsthm}
\usepackage[colorlinks=true,linkcolor=blue,citecolor=blue,%
filecolor=blue,urlcolor=blue,pageanchor=true,plainpages=false,%
linktocpage]{hyperref}

\numberwithin{equation}{section}
\newtheorem{thm}{Theorem}[section]
\newtheorem{prop}[thm]{Proposition}

\newtheorem{lem}[thm]{Lemma}
\theoremstyle{definition}
\newtheorem{defn}[thm]{Definition}
\newtheorem{rem}[thm]{Remark}

\newcommand{\Z}{\mathbb{Z}}
\newcommand{\R}{\mathbb{R}}
\newcommand{\G}{\mathbb{G}}

\newcommand{\dvg}{\mathrm{div}\,}
\newcommand{\w}{W^{1,p}_0(\Omega)}

\renewcommand{\rho}{\varrho}
\renewcommand{\theta}{\vartheta}
\newcommand{\idx}[1]{\operatorname{\mathrm{Index}}\left(#1\right)}

\begin{document}


\title[Amann-Zehnder type results for $p$-Laplace 
problems]{Amann-Zehnder type results for $p$-Laplace problems}

\author{Silvia Cingolani}
\address{Dipartimento di Meccanica, Matematica e Management\\
         Politecnico di Bari\\
         Via Ora\-bo\-na 4\\
         70125 Bari, Italy}
\email{silvia.cingolani@poliba.it}
\thanks{The research of the authors was partially supported by 
        Gruppo Nazionale per l'Analisi Matematica, la Probabilit\`a
				e le loro Applicazioni (INdAM). 
				In particular S.~Cingolani is supported
        by G.N.A.M.P.A. Project 2016
        ``Studio variazionale di fenomeni fisici non lineari''.}
\author{Marco Degiovanni}
\address{Dipartimento di Matematica e Fisica\\
         Universit\`a Cattolica del Sacro Cuore\\
         Via dei Musei 41\\
         25121 Bre\-scia, Italy}
\email{marco.degiovanni@unicatt.it}
\author{Giuseppina Vannella}
\address{Dipartimento di Meccanica, Matematica e Management \\
         Politecnico di Bari\\
         Via Ora\-bo\-na 4\\
         70125 Bari, Italy}
\email{giuseppina.vannella@poliba.it}





%
\begin{abstract}
The existence of a nontrivial solution is proved for a class of
quasilinear elliptic equations involving, as principal part, 
either the $p$-Laplace operator or the operator related to the
$p$-area functional, and a nonlinearity with $p$-linear growth at 
infinity.
To this aim, Morse theory techniques are combined with
critical groups estimates.
\end{abstract}
\maketitle 


\section{Introduction}
In 1980,  Amann and Zehnder \cite{amann_zehnder} studied the
asymptotically linear elliptic problem
\begin{equation}
\label{eq:semi} 
\left\{
\begin{array}{ll}
- \, \Delta  u = g(u)
& \qquad\text{in $\Omega$} \,, \\
\noalign{\medskip}
u=0
& \qquad\text{on $\partial\Omega$} \,,
\end{array}
\right.
\end{equation}
where  $\Omega$ is a bounded domain in $\R^N$ with smooth
boundary, $g:\R\rightarrow\R$ is a $C^1$-function such that 
$g(0)=0$ and there exists $\lambda\in\R$ such that
\[
\lim_{|s|\to\infty} g'(s) = \lambda\,.
\]
They proved that problem~\eqref{eq:semi} admits a nontrivial 
solution $u$, supposing that $\lambda$ is not an eigenvalue of 
$-\Delta$, the so-called nonresonance condition at infinity, and 
that there exists some eigenvalue of $-\Delta$ between $\lambda$ 
and $g'(0)$.
The same result was obtained by Chang~\cite{chang1981-cpam} in 
1981, using Morse theory for manifolds with boundary, and by  
Lazer and Solimini~\cite{lazer_solimini1988} in 1988, combining 
mini-max characterization of the critical point and Morse index 
estimates.
More precisely, the basic idea in \cite{lazer_solimini1988} is to 
recognize that the energy functional associated to the 
asymptotically linear problem~\eqref{eq:semi} has a saddle geometry, 
which implies that a suitable Poincar\'e polynomial is not trivial, 
and also to show that a certain critical group at zero is trivial, 
to ensure the existence of a solution $u \neq 0$ of~\eqref{eq:semi}.
\par
In the present work, we are interested in finding nontrivial 
solutions~$u$ for the quasilinear elliptic problem
\begin{equation}
\label{eq:quasmod} 
\left\{
\begin{array}{ll}
- \, \dvg\left[\left(\kappa^2+
|\nabla u|^2\right)^{\frac{p-2}2} \nabla u\right] = g(u)
& \qquad\text{in $\Omega$} \,, \\
\noalign{\medskip}
u=0
& \qquad\text{on $\partial\Omega$} \,,
\end{array}
\right.
\end{equation}
where $\Omega$ is a bounded open subset of $\R^N$, $N\geq 1$,
with $\partial\Omega$ of class $C^{1,\alpha}$ for some
$\alpha\in]0,1]$, while $\kappa\geq 0$, $p >1$ are real numbers, 
and $g:\R\rightarrow\R$ is a $C^1$-function such that:
\begin{itemize}
\item[$(a)$]
$g(0)=0$ and there exists $\lambda\in\R$ such that
\[
\lim_{|s|\to\infty} \frac{g(s)}{|s|^{p-2}s} = \lambda\,.
\]
\end{itemize}
About the principal part of the equation, the reference cases
are $\kappa=0$, which yields the \emph{$p$-Laplace operator},
and $\kappa=1$, which yields the operator related to the 
\emph{$p$-area functional}.
In the case $p=2$ the value of $\kappa$ is irrelevant.
\par
It is standard that weak solutions $u$ of~\eqref{eq:quasmod} 
correspond to critical points of the
$C^1$-functional $f:W^{1,p}_0(\Omega)\rightarrow\R$ defined as
\begin{equation}
\label{eq:fmod} 
f(u) = 
\int_{\Omega} \Psi_{p,\kappa}(\nabla u)\,dx
- \int_{\Omega} G(u)\,dx \,,
\end{equation}
where
\[
\Psi_{p,\kappa}(\xi) = \frac{1}{p}\,
\left[\left(\kappa^2 + |\xi|^2\right)^{\frac{p}{2}}
- \kappa^p\right]\,,
\qquad G(s)= \int_0^s g(t) dt.
\]\par
With reference to the approach of~\cite{lazer_solimini1988}, 
when $p \neq 2$ the new difficulties that one has to face 
are related to both the main ingredients of the argument, 
namely to recognize a saddle structure, with a related 
information on a suitable Poincar\'e polynomial, and to 
provide an estimate of the critical groups at zero by some 
Hessian type notion.
\par
Concerning the first aspect, the spectral properties of 
$-\Delta_p$ are not yet well understood. 
We say that the real number $\lambda$ is an eigenvalue of 
$-\Delta_p$ if the equation $-\Delta_p u = \lambda |u|^{p-2}u$ 
admits a nontrivial solution $u \in \w$ and we denote by 
$\sigma(-\Delta_p)$ the set of such eigenvalues.
It is known that there exists a first eigenvalue $\lambda_1>0$, 
which is simple, and a second eigenvalue $\lambda_2 >\lambda_1$, 
both possessing several equivalent characterizations
(see~\cite{anane1987, anane_tsuoli1996, cuesta2001, 
drabek_robinson1999, lindqvist1990}).
Moreover, one can define in at least three different ways a 
diverging sequence $(\lambda_m)$ of eigenvalues of $-\Delta_p$
(see~\cite{cingolani_degiovanni2005, drabek_robinson1999, 
perera2003-tmna}),
but it is not known if they agree for $m \geq 3$ and if  
the whole set $\sigma(-\Delta_p)$ is covered.
Therefore it is not standard to recognize a saddle type geometry 
for the energy functional associated to the quasilinear problem.
\par
On the other hand, for functionals defined on Banach spaces, 
serious difficulties arise in extending Morse theory
(see~\cite{uhlenbeck1972, tromba1977, chang1983-cam, chang1993,
chang1998}).
More precisely, by standard deformation results, which hold
also in general Banach spaces, one can prove the so-called
Morse relations, which can be written as
\[
\sum_{m=0}^{\infty} C_m t^m =
\sum_{m=0}^{\infty} \beta_m t^m + (1+t) Q(t)\,,
\]
where $(\beta_m)$ is the sequence of the Betti numbers of
a pair of sublevels $(\{f\leq b\},\{f < a\})$ and $(C_m)$
is a sequence related to the critical groups  of the
critical points $u$ of $f$ with $a\leq f(u) \leq b$
(see e.g. the next Definition~\ref{critgroups} 
and~\cite[Theorem~I.4.3]{chang1993}).
The problem, in the extension from Hilbert to Banach spaces,
concerns the estimate of $(C_m)$, hence of critical groups,
by the Hessian of $f$ or some related concept.
In a Hilbert setting, the classical Morse lemma and
the generalized Morse lemma~\cite{gromoll_meyer1969-t}
provide a satisfactory answer.
For Banach spaces, a similar general result is so far
not known, also due to the lack of Fredholm properties of 
the second derivative of the functional.
\par
The first difficulty has been overcome by the first two authors
in~\cite{cingolani_degiovanni2005} for a problem quite 
similar to~\eqref{eq:quasmod}.
By generalizing from~\cite{chang1993} the notion of homological 
linking, in~\cite[Theorem 3.6]{cingolani_degiovanni2005} an 
abstract result has been proved which allows to produce a pair 
of sublevels $(\{f\leq b\},\{f < a\})$ with a nontrivial homology
group. 
In order to describe its dimension in terms of $\lambda$
in the setting of problem~\eqref{eq:quasmod},
it is then convenient to set, whenever $m\geq 1$,
\[
\lambda_m = \inf \left\{ \sup_A \mathcal{E}:\,
\text{$A \subseteq M$, $A$ is symmetric and
$\idx{A}\geq m$} \right\} \,,
\]
where
\[
M = \left\{u \in W^{1,p}_0(\Omega):\, 
\int_{\Omega} |u|^p\,dx= 1\right\}\,,\qquad 
\mathcal{E}(u) = \int_{\Omega} |\nabla u|^p \,dx
\]
and $\mathrm{Index}$ denotes the $\Z_2$-cohomological index
of Fadell and Rabinowitz~\cite{fadell_rabinowitz1977,
fadell_rabinowitz1978}.
For a matter of convenience, we also set $\lambda_0=-\infty$.
It is well known that $(\lambda_m)$ is a nondecreasing divergent 
sequence.
The arguments of~\cite{cingolani_degiovanni2005} 
apply for any $p>1$.
\par
For the second difficulty, the value of $p$ becomes relevant.
In~\cite{cingolani_vannella2003-aihp} the first and the last 
author have proved, for $p>2$ and $\kappa>0$, an extension of the 
Morse Lemma and established a connection between the critical 
groups and the Morse index, taking advantage of the fact that, 
under suitable assumptions on $g$, the functional $f$ is 
actually of class $C^2$ on $W^{1,p}_0(\Omega)$ and that
\[
\Psi''_{p,\kappa}(\eta)[\xi]^2 \geq \nu_{p,\kappa}|\xi|^2
\qquad\text{with $\nu_{p,\kappa}>0$}\,.
\]
Related results in the line of Morse theory have been proved by 
the first and the last author, in the case $p>2$, in 
\cite{cingolani_vannella2003-bergamo, cingolani_vannella2006, 
cingolani_vannella2007}.
By means of the results of~\cite{cingolani_vannella2007},
an Amann-Zehnder type result has been proved 
in~\cite{cingolani_degiovanni2005} for a problem quite 
similar to~\eqref{eq:quasmod}, provided that $p>2$.
\par\smallskip
In this work we are first of all interested in a
corresponding result in the case $1<p<2$, which amounts
to establish a relation between critical groups and Hessian type 
notions also in this case.
Since our argument recovers also the case $p\geq 2$ with
less assumptions on $g$, we provide an Amann-Zehnder type  
result for any $p>1$.
\par
Let us point out that, if $1<p<2$, the functional $f$ is not 
of class $C^2$ on $W^{1,p}_0(\Omega)$.
For $\kappa=0$, even the function $\Psi_{p,\kappa}$ is
not of class $C^2$ on $\R^N$.
\par
If $\kappa>0$ or $p \geq 2$, let us denote by $m(f,0)$ the 
supremum of the dimensions of the linear subspaces where the 
quadratic form $Q_0:W^{1,2}_0(\Omega) \rightarrow\R$ defined as
\[
Q_0(u) =  
\begin{cases}
\displaystyle{
\kappa^{p-2}\,\int_{\Omega} |\nabla u|^2 \, dx
- g'(0) \, \int_{\Omega} u^2 \, dx}
&\text{if $\kappa>0$ or $p>2$}\,,\\
\noalign{\medskip}
\displaystyle{
\int_{\Omega} |\nabla u|^2 \, dx
- g'(0) \, \int_{\Omega} u^2 \, dx}
&\text{if $\kappa=0$ and $p=2$}\,,
\end{cases}
\]
is negative definite. 
Let us also denote by $m^*(f,0)$ the supremum of the dimensions 
of the linear subspaces where the quadratic form $Q_0$ is 
negative semidefinite.
If $\kappa=0$ and $1<p<2$, we set $m(f,0)=m^*(f,0)=0$.
\par
Our first result is the following:
\begin{thm}
\label{thm:nonres} 
Assume $1<p<\infty$, $\kappa\geq 0$ and hypothesis $(a)$ on $g$.	
Suppose also that  $\lambda\not\in\sigma(-\Delta_p)$ and
denote by $m_{\infty}$ the integer such that
$\lambda_{m_{\infty}} < \lambda < \lambda_{m_{\infty}+1}$.
\par
If 
\[	
m_{\infty} \not\in [m(f,0),m^*(f,0)]\,,
\]
then there exists a nontrivial solution $u$
of~\eqref{eq:quasmod}.
\end{thm}
It is easily seen that, if $p=2$, the assumption that there 
exists some eigenvalue of $-\Delta$ between $\lambda$ 
and $g'(0)$ is equivalent to
$m_{\infty} \not\in [m(f,0),m^*(f,0)]$.
\par
Differently from
\cite{cingolani_degiovanni2005}, we aim also to deal with the 
resonant case, namely $\lambda\in\sigma(-\Delta_p)$.
This is not motivated by the pure wish of facing a more 
complicated situation.
To our knowledge, nobody has so far excluded the possibility 
that $\sigma(-\Delta_p)=\{\lambda_1\}\cup[\lambda_2,+\infty[$. 
In such a case, the restriction
$\lambda\not\in\sigma(-\Delta_p)$ would be quite severe.
Taking into account Theorem~\ref{thm:nonres}, the next
result has interest if $\lambda\in\sigma(-\Delta_p)$.
\begin{thm}
\label{thm:res} 
Assume hypothesis $(a)$ on $g$ and one of the following:
\begin{itemize}
\item[$(b_-)$]
we have
\[
\lim_{|s|\to\infty} \left[pG(s) - g(s)s\right] = -\infty\,;
\]
then we denote by $m_{\infty}$ the integer such that
\[
\lambda_{m_{\infty}} < \lambda \leq \lambda_{m_{\infty}+1}\,;
\]
\item[$(b_+)$]
we have
\[
\lim_{|s|\to\infty} \left[pG(s) - g(s)s\right] = +\infty
\]
and, moreover, either $1<p\leq 2$ with $\kappa\geq 0$ or 
$p>2$ with $\kappa=0$;
then we denote by $m_{\infty}$ the integer such that
\[
\lambda_{m_{\infty}} \leq \lambda < \lambda_{m_{\infty}+1}\,.
\]
\end{itemize}
\indent
If 
\[	
m_{\infty} \not\in [m(f,0),m^*(f,0)]\,,
\]
then there exists a nontrivial solution $u$
of~\eqref{eq:quasmod}.
\end{thm}
\begin{rem}
Concerning the lower order term, examples of $g$
satisfying $(a)$ and $(b_+)$ or $(b_-)$ are given by
\begin{alignat*}{3}
&g(s) = \lambda(1+s^2)^{\frac{p-2}{2}}s
+ \mu(1+s^2)^{\frac{q-2}{2}}s
&&\qquad\text{with $\mu\neq 0$ and $0<q<p\leq 2$}\,,\\
&g(s) = \lambda |s|^{p-2}s
+ \mu |s|^{q-2} s
&&\qquad\text{with $\mu\neq 0$ and $2\leq q<p$}\,,
\end{alignat*}
so that, respectively,
\begin{alignat*}{3}
&G(s) = \frac{\lambda}{p}\,
\left[\left(1 + s^2\right)^{\frac{p}{2}} - 1\right] 
+ \frac{\mu}{q}\,
\left[\left(1 + s^2\right)^{\frac{q}{2}} - 1\right]\,, \\
&G(s) = \frac{\lambda}{p}\,|s|^p
+ \frac{\mu}{q}\,|s|^q\,.
\end{alignat*}
\end{rem}
\begin{rem}
Let $p=4$, so that
\[
\Psi_{p,\kappa}(\xi) 
= \frac{1}{4}\,|\xi|^4 + \frac{1}{2}\,\kappa^2 |\xi|^2\,,
\]
and let
\[
g(s) = \lambda_m s^3 + \mu s
\]
with $m\geq 1$ and $\mu>0$, so that
\[
\lim_{|s|\to\infty} \left[4G(s) - g(s)s\right] = +\infty\,,
\]
while
\[
f(u) = \frac{1}{4} \int_\Omega \left[
|\nabla u|^4 - \lambda_m |u|^4\right]\,dx 
+ \frac{1}{2} \int_\Omega \left[
\kappa^2|\nabla u|^2 - \mu |u|^2\right]\,dx\,.
\]
It is clear that we cannot describe the geometry
of the functional $f$, if we have no information
concerning $\kappa^2$ and $\mu$.
For this reason in~$(b_+)$ only the case
$\kappa=0$ is considered, when $p>2$.
\end{rem}
In Section~\ref{sect:critical groups} we state some results 
about the critical groups estimates for a large class of 
functionals including~\eqref{eq:fmod}.  
We refer to Theorems~\ref{thm:positivo}, \ref{thm:cornerstone},
\ref{thm:positivo2} for $k>0$ and to 
Theorems~\ref{thm:zero} and~\ref{thm:u=0} 
for $k=0$ and $1<p<2$ (such results have been announced, 
without proof, in~\cite{cingolani_degiovanni_vannella2014}). 
Moreover, in a more particular situation which is however
enough for the proof of Theorems~\ref{thm:nonres} 
and~\ref{thm:res}, we extend Theorem~\ref{thm:positivo} to any 
$\kappa\geq 0$ and $1<p<\infty$ (see Theorem~\ref{thm:general}).
\par
Sections~\ref{sect:aux}, \ref{sect:parmin}, \ref{sect:red} 
and~\ref{sect:proofcrit} are devoted to the proof, by a finite 
dimensional reduction introduced in a different setting 
in~\cite{lancelotti2002}, of the results stated 
in Section~\ref{sect:critical groups},
while~Section~\ref{sect:proofs} contains the proof of 
Theorems~\ref{thm:nonres} and~\ref{thm:res}.


\section{Critical groups estimates}
\label{sect:critical groups}
In this section we consider a class of functionals
including~\eqref{eq:fmod}.
More precisely, let $\Omega$ be a bounded open subset of 
$\R^N$, $N\geq 1$, with $\partial\Omega$ of class $C^{1,\alpha}$ 
for some $\alpha\in]0,1]$, and let
$f:W^{1,p}_0(\Omega)\to\R$ be the functional defined as
\begin{equation}
\label{eq:f}
f(u) = \int_{\Omega} \Psi(\nabla u)\,dx
- \int_{\Omega} G(x,u)\,dx
\end{equation}
where $G(x, s)= \int_0^s g(x, t) dt$.
We assume that:
\begin{itemize}
\item[$(\Psi_1)$]
the function $\Psi:\R^N\rightarrow\R$ is of class~$C^1$ 
with $\Psi(0)=0$ and $\nabla\Psi(0)=0$;
moreover, there exist $1<p<\infty$, $\kappa\geq 0$ and
$0<\nu\leq C$ such that the functions
$\left(\Psi-\nu\,\Psi_{p,\kappa}\right)$ and
$\left(C\,\Psi_{p,\kappa}-\Psi\right)$ are both convex;
such a $p$ is clearly unique;
\item[$(\Psi_2)$]
if $\kappa=0$ and $1<p<2$, then $\Psi$ is of class~$C^2$
on $\R^N\setminus\{0\}$;
otherwise, $\Psi$ is of class~$C^2$ on $\R^N$;
\item[$(g_1)$]
the function $g:\Omega\times\R\to\R$ is such that
$g(\cdot,s)$ is measurable for every $s\in\R$ and 
$g(x,\cdot)$ is of class~$C^1$ for a.e. $x\in\Omega$;
moreover, we suppose that:
\begin{itemize}
\item
if $p < N$, there exist $C, q > 0$ such that
$q \leq p^* -1=\frac{Np}{N-p} - 1$ and
\[
|g(x,s)| \leq  C( 1+ |s|^{q})
\qquad\text{for a.e. $x\in\Omega$ and every $s\in\R$}\,;
\]
\item
if $p = N$, there exist $C, q > 0$ such that
\[
|g(x,s)| \leq  C( 1+ |s|^{q})
\qquad\text{for a.e. $x\in\Omega$ and every $s\in\R$}\,;
\]
\item
if $p > N$, for every $S>0$ there exists $C_S>0$ such that
\[
|g(x,s)|\leq C_S
\qquad\text{for a.e. $x\in\Omega$ and every $s\in\R$
with $|s|\leq S$}\,;
\]
\end{itemize}
\item[$(g_2)$]
for every $S>0$ there exists $\widehat{C}_S>0$ such that
\[
|D_sg(x,s)|\leq \widehat{C}_S
\qquad\text{for a.e. $x\in\Omega$ and every $s\in\R$
with $|s|\leq S$}\,.
\]
\end{itemize}
From $(\Psi_1)$ it follows that $\Psi$ is strictly convex.
Moreover, under these assumptions, it is easily seen that
$f:W^{1,p}_0(\Omega)\rightarrow\R$ is of class~$C^1$,
while it is of class $C^2$ if $p>\max\{N,2\}$.
Finally, even in the case $g=0$,
$f$ is never of class~$C^2$ for $1<p<2$ and is of
class $C^2$ in the case $p=2$ iff $\Psi$ is a quadratic
form on $\R^N$
(see~\cite[Proposition~3.2]{abbondandolo_schwarz2009}).
\par
Now, let $u_0\in W^{1,p}_0(\Omega)$ be a critical point of
the functional $f$, namely a weak solution of
\[
\left\{
\begin{array}{ll}
- \, \dvg\left[\nabla\Psi(\nabla u)\right] = g(x,u)
& \qquad\text{in $\Omega$} \,, \\
\noalign{\medskip} u=0
& \qquad\text{on $\partial\Omega$} \,.
\end{array}
\right.
\]
According
to~\cite{guedda_veron1989, dibenedetto1983, lieberman1988,
	tolksdorf1983, tolksdorf1984},
$u_0 \in C^{1,\beta}(\overline\Omega)$ for some
$\beta\in]0,1]$ (see also the next
Theorems~\ref{thm:regLinfty} and~\ref{thm:regC1}).
\par
Let us recall the first ingredient we need
from~\cite{chang1993, degiovanni2011, mawhin_willem1989}.
\begin{defn}
\label{critgroups}
Let $\G$ be an abelian group, $c=f(u_0)$ and
$f^c= \left \{u \in \w\! : f(u) \leq c \right \}$.
The {\em $m$-th critical group of $f$ at $u_0$ with coefficients
in $\G$} is defined by
\[
C_m(f,u_0;\G) =
H^m \left(f^{c}, f^{c}\setminus \{u_0\};\G\right) \,,
\]
where $H^*$ stands for Alexander-Spanier
cohomology~\cite{spanier1966}.
We will simply write $C_m(f,u_0)$, if no confusion can arise.
\end{defn}
In general, it may happen that $C_m(f,u_0)$ is not finitely
generated for some $m$ and that $C_m(f,u_0)\neq\{0\}$ for 
infinitely many $m$'s.
If however $u_0$ is an isolated critical point, under
assumptions $(\Psi_1)$ and $(g_1)$ it follows
from~\cite[Theorem~1.1]{cingolani_degiovanni2009}
and~\cite[Theorem~3.4]{almi_degiovanni2013} that
$C_*(f,u_0)$ is of finite type.
\par
The other ingredient is a notion of Morse index,
which is not standard, as the functional~$f$ is
not in general of class $C^2$.
\par
In the case $\kappa>0$ and $1<p<\infty$, observe that
\begin{multline*}
\nu\,\min\left\{(p-1),1\right\}\, 
(\kappa^2+ |\eta|^2)^{\frac{p-2}{2}}\,|\xi|^2
\leq \Psi''(\eta)[\xi]^2 \\
\leq
C\,\max\left\{(p-1),1\right\}\,
(\kappa^2+ |\eta|^2)^{\frac{p-2}{2}}\,|\xi|^2
\qquad\text{for any $\eta, \xi\in\R^N$}\,,
\end{multline*}
as
$\left(\Psi-\nu\,\Psi_{p,\kappa}\right)$ and
$\left(C\,\Psi_{p,\kappa}-\Psi\right)$ are both convex.
Therefore, there exists
$\tilde{\nu}>0$ such that
\[
\tilde{\nu}\,|\xi|^2 \leq \Psi''(\nabla u_0(x))\,[\xi]^2
\leq\dfrac{1}{\tilde{\nu}}\,|\xi|^2
\qquad\text{for any $x\in\Omega$ and $\xi\in\R^N$}\,,
\]
as $\nabla u_0$ is bounded.
Moreover, $D_sg(x,u_0)\in L^{\infty}(\Omega)$,
as $u_0$ is bounded.
Thus, we can define a smooth quadratic form
$Q_{u_0}:W^{1,2}_0(\Omega) \rightarrow\R$ by
\[
Q_{u_0}(v) =
\int_{\Omega}\Psi''(\nabla u_0)[\nabla v]^2\,dx
- \int_{\Omega} D_sg(x,u_0)v^2\,dx
\]
and define the \emph{Morse index of $f$ at $u_0$}
(denoted by $m(f,u_0)$) as the supremum of  the dimensions of
the linear subspaces of $W^{1,2}_0(\Omega)$ where $Q_{u_0}$
is negative definite and the \emph{large Morse index of $f$
at $u_0$} (denoted by $m^*(f,u_0)$) as the supremum of
the dimensions of the linear subspaces of $W^{1,2}_0(\Omega)$
where $Q_{u_0}$ is negative semidefinite.
We clearly have $m(f,u_0)\leq m^*(f,u_0)<+\infty$.
Let us point out that $Q_{u_0}$ is well behaved on 
$W^{1,2}_0(\Omega)$, while~$f$ is naturally defined on
$W^{1,p}_0(\Omega)$.
\par
In the case $\kappa=0$ and $p>2$, we still have
\[
\Psi''(\nabla u_0(x))\,[\xi]^2
\leq\dfrac{1}{\tilde{\nu}}\,|\xi|^2
\qquad\text{for any $x\in\Omega$ and $\xi\in\R^N$}\,,
\]
so that 
$Q_{u_0}:W^{1,2}_0(\Omega) \rightarrow\R$, $m(f,u_0)$ and
$m^*(f,u_0)$ can be defined as before and
$m(f,u_0)\leq m^*(f,u_0)$.
However, $m(f,u_0)$ and $m^*(f,u_0)$ might take the 
value $+\infty$.
\par
Finally, in the case $\kappa=0$ and $1<p<2$ observe that
\[
\frac{(p-1)\nu}{|\eta|^{2-p}}\,|\xi|^2
\leq \Psi''(\eta)[\xi]^2 \leq
\frac{C}{|\eta|^{2-p}}\,|\xi|^2
\qquad\text{for any $\eta, \xi\in\R^N$ with $\eta\neq 0$}\,,
\]
whence
\[
\Psi''(\nabla u_0(x))\,[\xi]^2 \geq \tilde{\nu}\,|\xi|^2
\qquad\text{for any $x\in\Omega$ with $\nabla u_0(x)\neq 0$ 
and $\xi\in\R^N$}\,.
\]
Set
\begin{alignat*}{3}
&Z_{u_0}&&=\left\{x\in\Omega:\,\,
\nabla u_0(x) = 0\right\}\,,\\
&X_{u_0}&&=
\left\{v\in W^{1,2}_0(\Omega):\,\,\text{
	$\nabla v(x)=0$ a.e. in $Z_{u_0}$ and
	$\dfrac{|\nabla v|^2}{|\nabla u_0|^{2-p}}\in
	L^1(\Omega\setminus Z_{u_0})$}
\right\}\,.
\end{alignat*}
Then
\[
(v|w)_{u_0} = \int_{\Omega\setminus Z_{u_0}}
\Psi''(\nabla u_0)\,[\nabla v,\nabla w]\,dx
\]
is a scalar product on $X_{u_0}$ which
makes $X_{u_0}$ a Hilbert space continuously
embedded in~$W^{1,2}_0(\Omega)$.
Moreover, we can define a smooth quadratic form
$Q_{u_0}:X_{u_0}\rightarrow\R$ by
\[
Q_{u_0}(v) =
\int_{\Omega\setminus Z_{u_0}}
\Psi''(\nabla u_0)[\nabla v]^2\,dx
- \int_{\Omega} D_sg(x,u_0)v^2\,dx
\]
and denote again by $m(f,u_0)$ the supremum of
the dimensions of the linear subspaces of $X_{u_0}$
where $Q_{u_0}$ is negative definite and by $m^*(f,u_0)$ the
supremum of the dimensions of the linear subspaces of
$X_{u_0}$ where $Q_{u_0}$ is negative semidefinite.
Since the derivative of $Q_{u_0}$ is a compact
perturbation of the Riesz isomorphism, we 
have $m(f,u_0)\leq m^*(f,u_0)<+\infty$.
For a sake of uniformity, let us set
$X_{u_0}=W^{1,2}_0(\Omega)$ when $\kappa>0$ and $1<p<\infty$.
\par\smallskip
Now we can state the results concerning the critical groups 
estimates for the functional~\eqref{eq:f}.
\begin{thm}
\label{thm:positivo}
Let  $\kappa>0$ and $1<p<\infty$.
Let $u_0 \in W^{1,p}_0(\Omega)$ be a
critical point of the functional $f$ defined in~\eqref{eq:f}.
\par
Then we have
\[
C_m(f, u_0) = \{0\}\qquad\text{
whenever $m < m(f,u_0)$ or $m > m^*(f,u_0)$}\,.
\]
\end{thm}
When the quadratic form $Q_{u_0}$ has no kernel,
we can provide a complete description of the
critical groups.
\begin{thm}
\label{thm:cornerstone}
Let  $\kappa>0$ and $1<p<\infty$.
Let $u_0 \in W^{1,p}_0(\Omega)$ be a
critical point of the functional $f$ defined in~\eqref{eq:f}
with $m(f,u_0)=m^*(f,u_0)$.
\par
Then $u_0$ is an isolated critical point of $f$ and we have
\[
\left\{
\begin{array}{ll}
C_m(f,u_0) \approx \G
&\text{if $m = m(f,u_0)$}\,,\\
\noalign{\medskip}
C_m(f, u_0) = \{0\}
&\text{if $m \neq m(f,u_0)$}\,.
\end{array}
\right.
\]
\end{thm}
If $u_0$ is an isolated critical point of $f$, then
a sharper form of Theorem~\ref{thm:positivo} can be proved.
Taking into account Theorem~\ref{thm:cornerstone}, only the
case $m(f,u_0)<m^*(f,u_0)$ is interesting.
\begin{thm}
\label{thm:positivo2}
Let  $\kappa>0$ and $1<p<\infty$.
Let $u_0 \in W^{1,p}_0(\Omega)$ be an isolated
critical point of the functional $f$ defined in~\eqref{eq:f}
with $m(f,u_0)<m^*(f,u_0)$.
\par
Then one and only one of the following facts holds:
\begin{itemize}
\item[$(a)$]
we have
\[
\left\{
\begin{array}{ll}
C_m(f,u_0) \approx \G
&\text{if $m = m(f,u_0)$}\,,\\
\noalign{\medskip}
C_m(f, u_0) = \{0\}
&\text{if $m \neq m(f,u_0)$}\,;
\end{array}
\right.
\]
\item[$(b)$]
we have
\[
\left\{
\begin{array}{ll}
C_m(f,u_0) \approx \G
&\text{if $m = m^*(f,u_0)$}\,,\\
\noalign{\medskip}
C_m(f, u_0) = \{0\}
&\text{if $m \neq m^*(f,u_0)$}\,;
\end{array}
\right.
\]
\item[$(c)$]
we have
\[
C_m(f, u_0) = \{0\}\qquad\text{
whenever $m \leq m(f,u_0)$ or $m \geq m^*(f,u_0)$}\,.
\]
\end{itemize}
\end{thm}
\begin{rem}
\label{rem:p2}
Since the value of $\kappa$ is irrelevant in the case $p=2$,
Theorems~\ref{thm:positivo}, \ref{thm:cornerstone}
and~\ref{thm:positivo2}
cover also the case $\kappa=0$ with $p=2$.
\end{rem}
In the case $\kappa=0$ and $p\neq 2$, we cannot provide such 
a complete description.
Let us mention, however, that critical groups estimates  
have been obtained in~\cite{aftalion_pacella2004-tams}
when $\Omega$ is a ball centered at the origin, and the critical 
point $u_0$ is a positive and radial function such that
$|\nabla u_0( x)|\neq 0$  for $x \neq 0$.
\par
Apart from the radial case, if $p>2$ and $g$ is subjected to 
assumptions that imply~$f$ to be of class $C^2$ on 
$W^{1,p}_0(\Omega)$, it has been proved 
in~\cite[Theorem~3.1]{lancelotti2002} that 
$C_m(f, u_0) = \{0\}$ whenever $m < m(f,u_0)$.
On the contrary, there is no information, in general,
when $p>2$ and $m > m^*(f,u_0)$.
\par
In the case $1<p<2$, the situation turns out to be in some 
sense reversed.
We will prove a result when $m > m^*(f,u_0)$, while we have no 
information, in general, when $m < m(f,u_0)$. 
\begin{thm}
\label{thm:zero}
Let $\kappa=0$ and $1<p<2$.
Let $u_0 \in W^{1,p}_0(\Omega)$ be a
critical point of the functional $f$ defined in~\eqref{eq:f}.
\par
Then we have
\[
C_m(f, u_0) = \{0\}\qquad\text{
whenever $m > m^*(f,u_0)$}\,.
\]
\end{thm}
However, in the case $u_0=0$, we can provide an
optimal result in the line of 
Theorem~\ref{thm:cornerstone}.
\begin{thm}
\label{thm:u=0}
Let $\kappa=0$ and $1<p<2$.
Let $0$ be a critical point of the functional $f$ defined 
in~\eqref{eq:f}.
\par
Then we have $m(f,0)=m^*(f,0)=0$ and $0$ is a strict local 
minimum and an isolated critical point of $f$ with
\[
\left\{
\begin{array}{ll}
C_m(f,0) \approx \G
&\text{if $m = 0$}\,,\\
\noalign{\medskip}
C_m(f, 0) = \{0\}
&\text{if $m \neq 0$}\,.
\end{array}
\right.
\]
\end{thm}
Finally, under more specific assumptions we can extend 
Theorem~\ref{thm:positivo} to any $\kappa$ and $p$.
This will be enough for the results stated in the Introduction.
\par
\begin{thm}
\label{thm:general}
Let  $\kappa\geq 0$ and $1<p<\infty$.
Let $0$ be an isolated critical point of the functional $f$ 
defined in~\eqref{eq:f} and suppose that $g$ is independent of 
$x$ and satisfies assumption~$(g_1)$ with $q<p^*-1$ in the case
$p<N$.
\par
Then we have
\[
C_m(f, 0) = \{0\}\qquad\text{
whenever $m < m(f,0)$ or $m > m^*(f,0)$}\,.
\]
\end{thm}


\section{Some auxiliary results}
\label{sect:aux}
In the following, for any $q\in[1,\infty]$ we will
denote by $\|~\|_q$ the usual norm in $L^q(\Omega)$.
We also set, for any $u\in C^{1,\alpha}(\overline\Omega)$,
\[
\|u\|_{C^{1,\alpha}} =
\sup_{\Omega} |u| + \sup_{\Omega} |\nabla u| +
\sup_{\substack{x,y\in\Omega \\ x\neq y}}
\dfrac{|\nabla u(x)-\nabla u(y)|}{|x-y|^{\alpha}}\,.
\]
Throughout this section, we assume that $\Omega$ is a bounded
open subset of $\R^N$ and we suppose that $\Psi$ and $g$ 
satisfy assumptions~$(\Psi_1)$, $(\Psi_2)$ and~$(g_1)$,
without any further restriction on~$p$ and $\kappa$.
\par
In the first part, we adapt to our setting some regularity results
from~\cite{guedda_veron1989, dibenedetto1983, lieberman1988,
tolksdorf1983, tolksdorf1984}.
\begin{thm}
\label{thm:regLinfty}
For every $u_0\in W^{1,p}_0(\Omega)$, there exists $r>0$ such 
that, for any $u\in W^{1,p}_0(\Omega)$ and 
$w\in W^{-1,\infty}(\Omega)$ satisfying
\[
\left\{
\begin{array}{l}
\displaystyle{
\int_\Omega \left[
\nabla\Psi(\nabla u)\cdot\nabla v - g(x,u)v\right]\,dx 
\leq \langle w,v\rangle} \\
\noalign{\medskip}
\qquad\qquad\qquad\qquad
\text{for any $v\in W^{1,p}_0(\Omega)$ with
$vu\geq 0$ a.e. in $\Omega$}\,,\\
\noalign{\medskip}
\|\nabla u-\nabla u_0\|_p\leq r\,,
\end{array}
\right.
\]
we have $u\in L^{\infty}(\Omega)$ and
\[
\|u\|_{\infty} \leq
C\left(\|w\|_{W^{-1,\infty}}\right)\,.
\]
\end{thm}
\begin{proof}
We only sketch the proof the case $1<p<N$.
The case $p\geq N$ is similar and even simpler.
Since $\left(\Psi-\nu\Psi_{p,\kappa}\right)$ is convex, we have
\[
\nabla\Psi(\xi)\cdot\xi = (\nabla\Psi(\xi)-\nabla\Psi(0))\cdot\xi
\geq
\nu \left(\kappa^2 + |\xi|^2\right)^{\frac{p-2}{2}}\,|\xi|^2
\qquad\text{for every $\xi\in\R^N$}\,.
\]
Then the argument is the same
of~\cite[Corollary~1.1]{guedda_veron1989}.
We only have to remark that, for every $V_0 \in L^{N/p}(\Omega)$
and $q<\infty$, there exists $r>0$ such that,
for any $u\in W^{1,p}_0(\Omega)$, $V \in L^{N/p}(\Omega)$ 
and $w\in W^{-1,\infty}(\Omega)$ satisfying
\[
\left\{
\begin{array}{l}
\displaystyle{
\int_\Omega \left[
\nabla\Psi(\nabla u)\cdot\nabla v - V|u|^{p-2}u v\right]\,dx 
\leq \langle w,v\rangle} \\
\noalign{\medskip}
\qquad\qquad\qquad\qquad
\text{for any $v\in W^{1,p}_0(\Omega)$ with
$vu\geq 0$ a.e. in $\Omega$}\,,\\
\noalign{\medskip}
\|V-V_0\|_{N/p}\leq r\,,
\end{array}
\right.
\]
we have $u\in L^q(\Omega)$ and
\[
\|u\|_q \leq C\left(q,\|u\|_{p^*},\|w\|_{W^{-1,\infty}}\right)
\]
(see, in particular, 
\cite[Proposition~1.2 and Remark~1.1]{guedda_veron1989}).
The key point is that, for any $\varepsilon>0$, there exist
$r,\overline k$ such that
\[
k\geq \overline k\quad\Longrightarrow\quad
\int_{\{|V|>k\}} |V|^{N/p}\,dx \leq \varepsilon
\]
whenever $\|V-V_0\|_{N/p}\leq r$.
After removing the dependence on $V$
in~\cite[Proposition~1.2]{guedda_veron1989}, hence on $u$
in~\cite[Corollary~1.1]{guedda_veron1989}, the argument is
the same of~\cite{guedda_veron1989}.
\end{proof}
\begin{thm}
\label{thm:regC1}
Assume that $\partial\Omega$ is of class $C^{1,\alpha}$ for some
$\alpha\in]0,1]$.
Then there exists $\beta\in]0,1]$ such that any solution $u$ of
\[
\left\{
\begin{array}{ll}
u\in W^{1,p}_0(\Omega)\,,\\
\noalign{\medskip}
- \, \dvg[\nabla\Psi(\nabla u)] = w_0 - \dvg w_1
&\qquad\text{in $W^{-1,p'}(\Omega)$}\,,
\end{array}
\right.
\]
with $w_0\in L^{\infty}(\Omega)$
and $w_1\in C^{0,\alpha}(\overline\Omega;\R^N)$,
belongs also to $C^{1,\beta}(\overline\Omega)$ and we have
\[
\|u\|_{C^{1,\beta}} \leq
C\left(\|w_0\|_{\infty},\|w_1\|_{C^{0,\alpha}}\right)\,.
\]
\end{thm}
\begin{proof}
Since $\Psi$ is strictly convex and
\[
\nabla\Psi(\xi)\cdot\xi  \geq
\nu \left(\kappa^2 + |\xi|^2\right)^{\frac{p-2}{2}}\,|\xi|^2
\qquad\text{for every $\xi\in\R^N$,}
\]
it is standard that, for every $w_0\in L^{\infty}(\Omega)$
and $w_1\in C^{0,\alpha}(\overline\Omega;\R^N)$,
there exists one and only one
$u\in W^{1,p}_0(\Omega)\cap L^{\infty}(\Omega)$ such that
$- \, \dvg[\nabla\Psi(\nabla u)] = w_0 - \dvg w_1$.
Moreover, we have
\[
\|u\|_{\infty} \leq
C\left(\|w_0\|_{\infty},\|w_1\|_{C^{0,\alpha}}\right)
\]
(see e.g.~\cite{ladyzhenskaya_uraltseva1968}).
\par
Now, for every $N\geq 1$, fix a nonnegative smooth function
$\varrho$ with compact support in the unit ball of $\R^N$ and
unit integral.
Then define, for every $\Phi\in L^1_{loc}(\R^N)$ and
$\varepsilon>0$,
\[
(R_{\varepsilon}\Phi)(\xi) =
\int \varrho(y) \Phi(\xi-\varepsilon y) \,dy\,.
\]
It is easily seen that there exist
$0<\check{\nu}(N,p)\leq \check{C}(N,p)$ such that
\begin{gather*}
\check{\nu} (1+|\xi|)^{p-2} \leq
\int \varrho(y) (1+|\xi - t y|)^{p-2} \,dy \leq
\check{C} (1+|\xi|)^{p-2}\,, \\
\check{\nu} (1+|\xi|)^{p-2} \leq
\int \varrho(y) (t+|\xi - y|)^{p-2} \,dy \leq
\check{C} (1+|\xi|)^{p-2}\,,
\end{gather*}
for every $t\in[0,1]$ and $\xi\in\R^N$.
Then there exist
$0<\hat{\nu}(N,p)\leq \widehat{C}(N,p)$ such that
\[
\hat{\nu} (\varepsilon+\kappa+|\xi|)^{p-2} \leq
\int \varrho(y) (\kappa+|\xi - \varepsilon y|)^{p-2}
\,dy \leq
\widehat{C} (\varepsilon+\kappa+|\xi|)^{p-2}\,,
\]
for every $\kappa\geq 0$, $\varepsilon>0$ and $\xi\in\R^N$.
\par
Observe that
$\Psi_{p,\kappa} \in W^{2,1}_{loc}(\R^N)$ and
\begin{multline*}
\frac{p-1}{2}\,
\left(\kappa+|\eta|\right)^{p-2}\,|\xi|^2\leq
\Psi''_{p,\kappa}(\eta)[\xi]^2
\leq
\left(\kappa+|\eta|\right)^{p-2}\,|\xi|^2 \\
\qquad\text{for every $\eta, \xi\in\R^N$ with $\eta\neq 0$.}
\end{multline*}
It follows that there exist
$0<\tilde{\nu}(N,p)\leq \widetilde{C}(N,p)$ such that
\[
\tilde{\nu} (\varepsilon+\kappa+|\eta|)^{p-2} |\xi|^2 \leq
(R_{\varepsilon}\Psi_{p,\kappa})''(\eta)[\xi]^2
\leq
\widetilde{C} (\varepsilon+\kappa+|\eta|)^{p-2} |\xi|^2\,,
\]
for every $\kappa\geq 0$, $\varepsilon>0$ and $\xi,\eta\in\R^N$.
Since $\left(\Psi-\nu\,\Psi_{p,\kappa}\right)$ and
$\left(C\,\Psi_{p,\kappa}-\Psi\right)$ are both convex,
we infer that $R_{\varepsilon}\Psi:\R^N\rightarrow \R$ is a
smooth function satisfying
\begin{equation}
\label{eq:RPsi''}
\nu \tilde{\nu} (\varepsilon+\kappa+|\eta|)^{p-2} |\xi|^2
\leq
(R_{\varepsilon}\Psi)''(\eta)[\xi]^2
\leq
C \widetilde{C} (\varepsilon+\kappa+|\eta|)^{p-2} |\xi|^2\,,
\end{equation}
for every $\varepsilon>0$ and $\xi,\eta\in\R^N$.
\par
Again, from the results of~\cite{ladyzhenskaya_uraltseva1968},
it follows that there exists one and only one
$u_{\varepsilon}\in W^{1,p}_0(\Omega)\cap L^{\infty}(\Omega)$
such that
$-\,\dvg[\nabla(R_{\varepsilon}\Psi)(\nabla u_{\varepsilon})] =
w_0 - \dvg w_1$.
Moreover, we have
\[
\|u_{\varepsilon}\|_{\infty} \leq
C\left(\|w_0\|_{\infty},\|w_1\|_{C^{0,\alpha}}\right)
\]
and the estimate is independent of $\varepsilon$ for, say,
$0<\varepsilon\leq 1$.
\par
Then from~\eqref{eq:RPsi''} and~\cite[Theorem~1]{lieberman1988} we
infer that $u_{\varepsilon}\in C^{1,\beta}(\overline{\Omega})$ and
\[
\|u_{\varepsilon}\|_{C^{1,\beta}} \leq
C\left(\|w_0\|_{\infty},\|w_1\|_{C^{0,\alpha}}\right)
\]
for some $\beta\in]0,1]$, again with
an estimate independent of $\varepsilon\in ]0,1]$.
\par
Therefore $(u_{\varepsilon})$ is convergent, as $\varepsilon\to 0$,
to $u$ in $C^1(\overline{\Omega})$ and the assertion follows.
\end{proof}
From~\eqref{eq:RPsi''} we also infer the next result.
\begin{prop}
\label{prop:Psireg}
We have $\Psi\in W^{2,q}_{loc}(\R^N)$ for some $q>N$,
so that the map $\nabla\Psi:\R^N\rightarrow\R^N$ is locally
H\"older continuous.
\end{prop}
\par\smallskip
Now let $X$ be a reflexive Banach space.
The next concept is taken from~\cite{browder1983, skrypnik1994}.
\begin{defn}
\label{defn:S+}
Let $D \subseteq X$.
A map $F:D \rightarrow X'$ is said to be \emph{of class~$(S)_+$}
if, for every sequence $(u_k)$ in $D$ weakly convergent to $u$
in $X$ with
\[
\limsup_k \,\langle F(u_k),u_k-u\rangle \leq 0\,,
\]
we have $\|u_k-u\|\to 0$.
\end{defn}
\begin{prop}
\label{prop:S+}
Let $f:X\rightarrow \R$ be a function of class~$C^1$and let $C$
be a closed and convex subset of $X$.
Assume that $f'$ is of class~$(S)_+$ on~$C$.
\par
Then the following facts hold:
\begin{itemize}
\item[$(a)$]
$f$ is sequentially lower semicontinuos on $C$
with respect to the weak topology;
\item[$(b)$]
if $(u_k)$ is a sequence in $C$ weakly convergent to $u$ with
\[
\limsup_k \,f(u_k) \leq f(u)\,,
\]
we have $\|u_k-u\|\to 0$;
\item[$(c)$]
any bounded sequence $(u_k)$ in $C$, with
$\|f'(u_k)\|\to 0$, admits a convergent subsequence.
\end{itemize}
\end{prop}
\begin{proof}
Let $(u_k)$ be a sequence in $C$ weakly convergent to $u$.
To prove~$(a)$ we may assume, without loss of generality, that
\[
\limsup_k\, f(u_k) \leq f(u)\,.
\]
Let $t_k\in]0,1[$ be such that
\[
f(u_k) = f(u) + \langle f'(v_k),u_k-u\rangle\,,
\qquad v_k=u+t_k(u_k-u)\,.
\]
Then $(v_k)$ also is a sequence in $C$ weakly convergent
to $u$ and
\[
\limsup_k\,\langle f'(v_k),v_k-u\rangle =
\limsup_k\,t_k\langle f'(v_k),u_k-u\rangle =
\limsup_k\,t_k\left(f(u_k)-f(u)\right) \leq 0 \,.
\]
Since $f'$ is of class~$(S)_+$ on $C$, we infer that
$\|v_k-u\|\to 0$, hence that
\[
\lim_k f(u_k) =
\lim_k \left[f(u) + \langle f'(v_k),u_k-u\rangle\right]
= f(u)
\]
and assertion~$(a)$ follows.
\par
To prove~$(b)$, let $\tau_k\in\left]\frac{1}{2},1\right[$
be such that
\[
f(u_k) - f\left(\frac{1}{2}\,u_k+\frac{1}{2}\,u\right) =
\frac{1}{2}\,\langle f'(w_k),u_k-u\rangle \,,
\qquad w_k = u+\tau_k(u_k-u)\,.
\]
Observe that
$\left(\frac{1}{2}\,u_k+\frac{1}{2}\,u\right)$ also is a sequence
in $C$ weakly convergent to $u$, whence
\[
\liminf_k f\left(\frac{1}{2}\,u_k+\frac{1}{2}\,u\right)
\geq f(u)\,.
\]
It follows
\[
\limsup_k\,\langle f'(w_k),u_k-u\rangle
= \limsup_k\,2\left[f(u_k)-
f\left(\frac{1}{2}\,u_k+\frac{1}{2}\,u\right)\right]\leq 0
\]
whence, as before, $\|w_k-u\|\to 0$.
Since $(\tau_k)$ is bounded away from $0$, we conclude that
$\|u_k-u\|\to 0$.
\par
Finally, to prove~$(c)$ we may assume that
$(u_k)$ is weakly convergent to some $u$, whence
\[
\lim_k \,\langle f'(u_k),u_k-u\rangle = 0\,.
\]
Since $f'$ is of class~$(S)_+$ on $C$, assertion $(c)$ also follows.
\end{proof}
We end the section with a result relating the minimality in the 
$C^1$-topology and that in the  $W^{1,p}_0$-topology.
When $W=W^{1,p}_0(\Omega)$ and $\Psi(\xi)=\frac{1}{p}|\xi|^p$, 
the next theorem has been proved
in~\cite{garciaazorero_peralalonso_manfredi2000}, which extends to
the $p$-Laplacian the well-known result by
Brezis and Nirenberg~\cite{brezis_nirenberg1993} for the case $p=2$
(see also~\cite{guo_zhang2003} for $p>2$
and~\cite{kyritsi_papageorgiou2005} in a non-smooth setting).
\begin{thm}
\label{thm:locminW1p}
Assume that $\partial\Omega$ of class $C^{1,\alpha}$ and that
$u_0\in W^{1,p}_0(\Omega)\cap C^{1,\alpha}(\overline{\Omega})$
for some $\alpha\in]0,1]$.
Suppose also that $W^{1,p}_0(\Omega)=V\oplus W$, where $V$ is a 
finite dimensional subspace of $W^{1,p}_0(\Omega)$,
$W$ is closed in $W^{1,p}_0(\Omega)$ and the projection
$P_V:W^{1,p}_0(\Omega)\rightarrow V$, associated with
the direct sum decomposition, is continuous from the topology of
$L^1(\Omega)$ to that of $V$.
\par
If $u_0$ is a strict local minimum for the functional $f$ defined 
in~\eqref{eq:f} along $u_0+(W\cap C^1(\overline{\Omega}))$ for the 
$C^1(\overline{\Omega})$-topology, then $u_0$ is a strict local 
minimum of $f$ along $u_0+W$ for the $W^{1,p}_0(\Omega)$-topology.
\end{thm}
\begin{proof}
Define a convex and coercive functional 
$h:W^{1,p}_0(\Omega)\rightarrow\R$ by
\[
h(u) = \int_{\Omega}\bigl[
\Psi(\nabla u) - \Psi(\nabla u_0)
- \nabla\Psi(\nabla u_0)\cdot
(\nabla u-\nabla u_0)\bigr]\,dx
\]
and observe that $v_k \to u_0$ in $W^{1,p}_0(\Omega)$
if and only if $h(v_k) \to 0$.
Actually, if $v_k \to u_0$ in $W^{1,p}_0(\Omega)$, it is
clear that $h(v_k) \to 0$.
Conversely, assume that $h(v_k) \to 0$.
Since
\[
\Psi(\nabla v_k) - \Psi(\nabla u_0)
- \nabla\Psi(\nabla u_0)\cdot
(\nabla v_k-\nabla u_0) \to 0
\qquad\text{in $L^1(\Omega)$}\,,
\]
up to a subsequence we have
\[
\Psi(\nabla v_k) - \Psi(\nabla u_0)
- \nabla\Psi(\nabla u_0)\cdot
(\nabla v_k-\nabla u_0) \to 0
\qquad\text{a.e. in $\Omega$}\,,
\]
hence $\nabla v_k\to \nabla u_0$ a.e. in $\Omega$ by
the strict convexity of $\Psi$.
On the other hand,
\begin{multline*}
\Psi(\nabla v_k) - \Psi(\nabla u_0)
- \nabla\Psi(\nabla u_0)\cdot
(\nabla v_k-\nabla u_0) \\
\geq \dfrac{\nu}{p}\,|\nabla v_k|^p
- \nabla\Psi(\nabla u_0)\cdot\nabla v_k - z
\qquad\text{a.e. in $\Omega$}
\end{multline*}
for some $z\in L^1(\Omega)$.
Therefore, $(\nabla v_k)$ is convergent to $\nabla u_0$
also weakly in $L^p(\Omega)$.
If we apply Fatou's Lemma to the sequence
\[
\left[\Psi(\nabla v_k) - \Psi(\nabla u_0)
- \nabla\Psi(\nabla u_0)\cdot
(\nabla v_k-\nabla u_0)\right] -
\dfrac{\nu}{p}\,|\nabla v_k|^p
+ \nabla\Psi(\nabla u_0)\cdot\nabla v_k + z \geq 0\,,
\]
we find that
\[
\limsup_k \int_{\Omega}|\nabla v_k|^p\,dx
\leq \int_{\Omega}|\nabla u_0|^p\,dx\,,
\]
whence the convergence of $(v_k)$ to $u_0$
in $W^{1,p}_0(\Omega)$.
\par
Since $h$ is of class $C^1$ with
\[
\langle h'(u),u-u_0\rangle = \int_{\Omega}\bigl(
\nabla\Psi(\nabla u) - \nabla\Psi(\nabla u_0)\bigr)
\cdot(\nabla u -\nabla u_0)\,dx>0
\qquad\text{for any $u\neq u_0$}\,,
\]
for every $r>0$ the set
\[
\left\{w\in W:\,\,h(u_0+w) = r\right\}
\]
is a $C^1$-hypersurface in $W$.
Moreover, if $r$ is small enough, the map $f'$ is of 
class~$(S)_+$ on 
\[
\left\{u\in W^{1,p}_0(\Omega):\,\,h(u_0+u)\leq r\right\}\,.
\]
If $\Psi=\Psi_{p,0}$ with $1<p<N$, this is proved
in~\cite[Theorem~1.2]{cingolani_degiovanni2009},
while the general case follows 
from~\cite[Theorem~3.4]{almi_degiovanni2013}.
From Proposition~\ref{prop:S+} we infer that
$\{u\mapsto f(u_0+u)\}$ is weakly lower
semicontinuous on
\[
\left\{u\in W^{1,p}_0(\Omega):\,\,h(u_0+u)\leq r\right\}\,,
\]
hence on
\[
\left\{w\in W:\,\,h(u_0+w)\leq r\right\}\,,
\]
which is weakly compact.
\par
If we argue by contradiction, we find a sequence
$(w_k)$ in $W$ such that $w_k$ is a minimum of
$\{w\mapsto f(u_0+w)\}$ on
\[
\left\{w\in W:\,\,h(u_0+w)\leq r_k\right\}
\]
with $r_k\to 0$ and $w_k\neq 0$, 
in particular $f(u_0+w_k) \leq f(u_0)$.
Therefore, there exists $\lambda_k\geq 0$ such that
\[
\langle f'(u_0+w_k) + \lambda_k h'(u_0+w_k),u\rangle=0
\qquad\text{for any $u\in W$}\,,
\]
namely
\begin{multline*}
\int_{\Omega}
\nabla\Psi(\nabla (u_0+w_k))\cdot\nabla u\,dx
- \int_{\Omega} g(x,u_0+w_k)u\,dx \\
+ \lambda_k
\int_{\Omega}\bigl(
\nabla\Psi(\nabla (u_0+w_k)) - \nabla\Psi(\nabla u_0)\bigr)
\cdot\nabla u\,dx = 0
\qquad\text{for any $u\in W$}\,,
\end{multline*}
which is equivalent to
\begin{multline*}
\int_{\Omega}
\nabla\Psi(\nabla (u_0+w_k))\cdot\nabla u\,dx
- \dfrac{1}{1+\lambda_k}\,\int_{\Omega} g(x,u_0+w_k)u\,dx \\
= \dfrac{\lambda_k}{1+\lambda_k}\,
\int_{\Omega}\nabla\Psi(\nabla u_0)
\cdot\nabla u\,dx
\qquad\text{for any $u\in W$}\,.
\end{multline*}
It follows
\begin{multline*}
\int_{\Omega}
\nabla\Psi(\nabla (u_0+w_k))\cdot\nabla u\,dx
- \dfrac{1}{1+\lambda_k}\,\int_{\Omega} g(x,u_0+w_k)u\,dx
= \dfrac{\lambda_k}{1+\lambda_k}\,
\int_{\Omega}\nabla\Psi(\nabla u_0)
\cdot\nabla u\,dx \\
+ \int_{\Omega}\nabla\Psi(\nabla (u_0+w_k))\cdot\nabla P_Vu\,dx
- \dfrac{1}{1+\lambda_k}\,\int_{\Omega} g(x,u_0+w_k)P_Vu\,dx \\
- \dfrac{\lambda_k}{1+\lambda_k}\,
\int_{\Omega}\nabla\Psi(\nabla u_0)
\cdot\nabla P_Vu\,dx
\qquad\text{for any $u\in W^{1,p}_0(\Omega)$}\,.
\end{multline*}
Since $P_V$ is continuous from the topology of
$L^1(\Omega)$ to that of $W^{1,p}_0(\Omega)$, we have
\begin{multline*}
\int_{\Omega}\nabla\Psi(\nabla (u_0+w_k))\cdot\nabla P_Vu\,dx
- \dfrac{1}{1+\lambda_k}\,\int_{\Omega} g(x,u_0+w_k)P_Vu\,dx \\
- \dfrac{\lambda_k}{1+\lambda_k}\,
\int_{\Omega}\nabla\Psi(\nabla u_0)
\cdot\nabla P_Vu\,dx
= \int_{\Omega} z_k u\,dx
\qquad\text{for any $u\in W^{1,p}_0(\Omega)$}
\end{multline*}
with $(z_k)$ bounded in $L^{\infty}(\Omega)$.
It follows
\begin{multline}
\label{eq:zk}
- \, \dvg\left[
\nabla\Psi(\nabla (u_0+w_k))\right]
- \dfrac{1}{1+\lambda_k}\,g(x,u_0+w_k)
= z_k -
\dvg\left[\dfrac{\lambda_k}{1+\lambda_k}\,
\nabla\Psi(\nabla u_0)\right]
\end{multline}
and $\nabla\Psi(\nabla u_0)\in C^{0,\beta}(\overline\Omega;\R^N)$
for some $\beta\in]0,1]$, by Proposition~\ref{prop:Psireg}.
\par
If $p<N$, from~$(g_1)$ we infer that
\[
\begin{split}
\dfrac{1}{1+\lambda_k}\, 
\left[g(x,u_0+w_k) u - g(x,0)\right] u
&\leq
|g(x,u_0+w_k)-g(x,0)|\,|u| \\
&=
\frac{|g(x,u_0+w_k)-g(x,0)|}{|u_0+w_k|}
\,(u_0+w_k) u \\
&\leq
\frac{C(1+|u_0+w_k|^{p^*-1})}{|u_0+w_k|}
\,(u_0+w_k)u \\
&=
C\frac{u_0+w_k}{|u_0+w_k|}\,u 
+ C|u_0+w_k|^{p^*-2}(u_0+w_k)u \,,
\end{split}
\]
whenever $u(u_0+w_k)\geq 0$ a.e. in $\Omega$.
It follows
\begin{multline*}
\int_{\Omega}\nabla\Psi(\nabla (u_0+w_k))\cdot\nabla u\,dx
- \int_{\Omega} C|u_0+w_k|^{p^*-2}(u_0+w_k)u\,dx \\
\leq \int_{\Omega} \left[
\dfrac{1}{1+\lambda_k}\, g(x,0)
+ \hat{z}_k+z_k\right] u\,dx
+ \dfrac{\lambda_k}{1+\lambda_k}\,
\int_{\Omega}\nabla\Psi(\nabla u_0)\cdot\nabla u\,dx \\
\qquad\text{for any $u\in W^{1,p}_0(\Omega)$
with $u(u_0+w_k)\geq 0$ a.e. in $\Omega$}\,,
\end{multline*}
where
\[
\hat{z}_k=\left\{
\begin{array}{ll}
\displaystyle{
C\frac{u_0+w_k}{|u_0+w_k|}} &\text{where $u_0+w_k\neq 0$}\,,\\
0 &\text{where $u_0+w_k=0$}\,.
\end{array}
\right. 
\]
From Theorem~\ref{thm:regLinfty} it follows that
$(u_0+w_k)$ is bounded in $L^\infty(\Omega)$.
Coming back to~\eqref{eq:zk},
from Theorem~\ref{thm:regC1}
we conclude that $(u_0+w_k)$ is bounded
in $C^{1,\beta}(\overline\Omega)$ for some
$\beta\in ]0,1]$.
Then $(u_0+w_k)$ is convergent to $u_0$ in
$C^1(\overline\Omega)$ and a contradiction follows.
\par
If $p\geq N$, the argument is similar and even simpler.
\end{proof}
%


\section{Parametric minimization}
\label{sect:parmin}
Throughout this section, we assume that $\Omega$ is a bounded open 
subset of $\R^N$ with $\partial\Omega$ of class $C^{1,\alpha}$
for some $\alpha\in]0,1]$ and that $\Psi$ and $g$ 
satisfy assumptions~$(\Psi_1)$, $(\Psi_2)$, $(g_1)$ and~$(g_2)$
with either 
$\kappa>0$ and $1<p<\infty$ or $\kappa=0$ and $1<p<2$.
\par
Let $u_0$ denote a critical point of the functional $f$ defined 
in~\eqref{eq:f}.
According to Theorems~\ref{thm:regLinfty} and~\ref{thm:regC1},
we have $u_0 \in C^{1,\beta}(\overline\Omega)$
for some $\beta\in]0,1]$.
\par
Given a continuous function $\Phi:\R^N\rightarrow\R$,
for any $x,v \in \R^N$ we set
\[
\underline{\Phi}''(x)[v]^2 =
\liminf_{\substack{y\to x\\ t\to 0\\ w\to v}}\,
\dfrac{\Phi(y+t w)+\Phi(y-t w)-2\Phi(y)}{t^2}\,.
\]
Then the function
$\left\{(x,v)\mapsto \underline{\Phi}''(x)[v]^2\right\}$
is lower semicontinuous.
If $\Phi$ is convex, it is also clear that
$\underline{\Phi}''(x)[v]^2\in[0,+\infty]$
and that $\underline{\Phi}''(x)[0]^2=0$.
In particular, it is easily seen that
\[
\text{$\kappa=0$ and $1<p<2$}\,\Longrightarrow\,
\underline{\Psi}''_{\,p,\kappa}(0)[\xi]^2 = \left\{
\begin{array}{ll}
0 &\qquad\text{if $\xi = 0$}\,,\\
\noalign{\medskip}
+\infty &\qquad\text{if $\xi\neq 0$}\,.
\end{array}
\right.
\]
Since $\left(\Psi-\nu\,\Psi_{p,\kappa}\right)$ is convex,
we also have
\[
\text{$\kappa=0$ and $1<p<2$}\,\Longrightarrow\,
\underline{\Psi}''(0)[\xi]^2 = \left\{
\begin{array}{ll}
0 &\qquad\text{if $\xi = 0$}\,,\\
\noalign{\medskip}
+\infty &\qquad\text{if $\xi\neq 0$}\,,
\end{array}
\right.
\]
while $\underline{\Psi}''(\eta)[\xi]^2=
\Psi''(\eta)[\xi]^2$ in the other cases.
In particular, the function
$\left\{\xi\mapsto \underline{\Psi}''(\eta)[\xi]^2\right\}$
is convex for any $\eta\in\R^N$:
\begin{prop}
\label{prop:PsiTaylor}
For every $u,v\in W^{1,p}_0(\Omega)$, the function
\[
\biggl\{(x,t) \mapsto (1-t)
\underline{\Psi}''
\bigl(\nabla u(x)+t(\nabla v(x)-\nabla u(x))\bigr)
\bigl[\nabla v(x)-\nabla u(x)\bigr]^2 \biggr\}
\]
belongs to $L^1(\Omega\times]0,1[)$ and one has
\begin{multline*}
\int_\Omega \Psi(\nabla v)\,dx -
\int_\Omega \Psi(\nabla u)\,dx -
\int_\Omega \nabla\Psi(\nabla u) \cdot(\nabla v-\nabla u)\,dx \\
= \int_0^1(1-t)\left\{\int_\Omega
\underline{\Psi}''
\bigl(\nabla u(x)+t(\nabla v(x)-\nabla u(x))\bigr)
\bigl[\nabla v(x)-\nabla u(x)\bigr]^2\, dx\right\}\,dt  \,.
\end{multline*}
\end{prop}
\begin{proof}
Let us treat the case $\kappa=0$ and $1<p<2$.
The case $\kappa>0$ and $1<p<\infty$ is similar and even simpler.
First of all,
$\left\{(\eta,\xi)\mapsto \underline{\Psi}''(\eta)[\xi]^2\right\}$
is a Borel function, being lower semicontinuous.
Moreover, we have
\[
\underline{\Psi}''(\eta)[\xi]^2=\Psi''(\eta)[\xi]^2 \leq
\frac{C}{|\eta|^{2-p}}\,|\xi|^2
\qquad\text{for any $\eta, \xi\in\R^N$ with $\eta\neq 0$}\,.
\]
Therefore, for every $\eta,\xi\in\R^N$, the function
$\left\{t\mapsto \Psi(\eta+t(\xi-\eta))\right\}$ belongs to
$W^{2,1}_{loc}(\R)$ and we have
\[
\Psi(\xi) - \Psi(\eta) - \nabla\Psi(\eta)\cdot(\xi-\eta) =
\int_0^1 (1-t)\underline{\Psi}''
\bigl(\eta+t(\xi-\eta)\bigr)
\bigl[\xi-\eta\bigr]^2\,dt\,.
\]
Then, given $u,v\in W^{1,p}_0(\Omega)$, we have a.e. in
$\Omega$
\begin{multline*}
\Psi(\nabla v(x)) - \Psi(\nabla u(x)) - \nabla\Psi(\nabla u(x))
\cdot(\nabla v(x)-\nabla u(x))  \\
=  \int_0^1 (1-t)\underline{\Psi}''
\bigl(\nabla u(x)+t(\nabla v(x)-\nabla u(x))\bigr)
\bigl[\nabla v(x)-\nabla u(x)\bigr]^2\,dt\,.
\end{multline*}
By integrating over $\Omega$ and applying Fubini's theorem, the
assertion follows.
\end{proof}
\begin{thm}
\label{thm:Qlsc}
Let $(u_k)$ be a sequence in
$W^{1,p}_0(\Omega)\cap L^{\infty}(\Omega)$
and $(v_k)$ a sequence in $W^{1,2}_0(\Omega)$
such that $(u_k)$ is bounded in $L^{\infty}(\Omega)$
and convergent to $u$ in $W^{1,p}_0(\Omega)$, while
$(v_k)$ is weakly convergent to $v$ in $W^{1,2}_0(\Omega)$.
\par
Then we have
\begin{multline*}
\int_{\Omega}
\underline{\Psi}''(\nabla u)[\nabla v]^2\,dx
- \int_{\Omega} D_sg(x,u)v^2\,dx \\
\leq \liminf_k \left(
\int_{\Omega}
\underline{\Psi}''(\nabla u_k)[\nabla v_k]^2\,dx
- \int_{\Omega} D_sg(x,u_k)v_k^2\,dx\right)\,.
\end{multline*}
\end{thm}
\begin{proof}
Since $(v_k)$ is convergent to $v$ in $L^2(\Omega)$, we clearly have
\[
\int_{\Omega} D_sg(x,u)v^2\,dx =
\lim_k \int_{\Omega} D_sg(x,u_k)v_k^2\,dx\,.
\]
Then the assertion follows from the Theorem in~\cite{ioffe1977-b}.
\end{proof}
\begin{prop}
\label{prop:decomposition}
There exists a direct sum decomposition
\[
L^1(\Omega) = V \oplus \widetilde{W}
\]
such that:
\begin{itemize}
\item[$(a)$]
$V\subseteq X_{u_0}\cap W^{1,p}_0(\Omega)\cap L^{\infty}(\Omega)$
with $\dim V = m^*(f,u_0)<+\infty$, while
$\widetilde{W}$ is closed in~$L^1(\Omega)$;
\item[$(b)$]
we have
\begin{align*}
&\int_{\Omega}
\underline{\Psi}''(\nabla u_0)[\nabla(v+w)]^2\,dx
- \int_{\Omega} D_sg(x,u_0)(v+w)^2\,dx \\
&\hskip50pt =\int_{\Omega}
\underline{\Psi}''(\nabla u_0)[\nabla v]^2\,dx
- \int_{\Omega} D_sg(x,u_0)v^2\,dx \\
&\hskip100pt + \int_{\Omega}
\underline{\Psi}''(\nabla u_0)[\nabla w]^2\,dx
- \int_{\Omega} D_sg(x,u_0)w^2\,dx \\
&\hskip175pt
\text{for any $v\in V$ and
$w\in \widetilde{W}\cap W^{1,2}_0(\Omega)$}\,,\\
&\int_{\Omega}
\underline{\Psi}''(\nabla u_0)[\nabla v]^2\,dx
- \int_{\Omega} D_sg(x,u_0)v^2\,dx \leq 0 \\
&\hskip175pt
\text{for any $v\in V$}\,,\\
&\int_{\Omega}
\underline{\Psi}''(\nabla u_0)[\nabla w]^2\,dx
- \int_{\Omega} D_sg(x,u_0)w^2\,dx > 0 \\
&\hskip175pt
\text{for any 
$w\in (\widetilde{W}\cap W^{1,2}_0(\Omega))\setminus\{0\}$}\,.
\end{align*}
\end{itemize}
\end{prop}
\begin{proof}
Let us treat in detail the case $\kappa=0$ and $1<p<2$.
Since the derivative of the smooth quadratic form
$Q_{u_0}:X_{u_0}\rightarrow\R$ is a compact perturbation of the
Riesz isomorphism, it is standard that there exists a direct sum
decomposition
\[
X_{u_0}  = V \oplus \widehat{W}
\]
such that $\dim V = m^*(f,u_0)<+\infty$,
\begin{alignat*}{3}
&\widehat{W} = \left\{w\in X_{u_0}:\,\,
\int_{\Omega} vw\,dx = 0\quad\text{for any $v\in V$}\right\}\,,\\
&Q_{u_0}(v+w) = Q_{u_0}(v) + Q_{u_0}(w)
&& \text{for any $v\in V$ and $w\in \widehat{W}$}\,,\\
&Q_{u_0}(v) \leq 0
&& \text{for any $v\in V$}\,,\\
&Q_{u_0}(w) > 0
&& \text{for any $w\in \widehat{W}\setminus\{0\}$}\,.
\end{alignat*}
Moreover, either $V=\{0\}$ or
$V=\mathrm{span}\left\{e_1,\ldots,e_{m^*}\right\}$
and each $e_j\in X_{u_0}\setminus\{0\}$ is a solution of
\[
\int_{\Omega\setminus Z_{u_0}}
\Psi''(\nabla u_0)[\nabla e_j,\nabla u]\,dx
- \int_{\Omega} D_sg(x,u_0)e_j u \,dx
= \lambda_j \int_{\Omega} e_j u\,dx
\qquad\text{for any $u\in X_{u_0}$}
\]
for some $\lambda_j\leq 0$ (which is possible only if
$\|\nabla u_0\|_{\infty}>0$).
\par
If $\varphi:\R\rightarrow\R$ is a nondecreasing Lipschitz function
with $\varphi(0)=0$, then $\varphi(e_j)\in X_{u_0}$, whence
\[
\int_{\Omega\setminus Z_{u_0}}
\varphi'(e_j)\Psi''(\nabla u_0)[\nabla e_j]^2\,dx
- \int_{\Omega} D_sg(x,u_0)e_j \varphi(e_j) \,dx
= \lambda_j \int_{\Omega} e_j \varphi(e_j)\,dx\leq 0\,.
\]
 On the other hand, we have
\begin{multline*}
\Psi''(\nabla u_0(x))[\xi]^2 \geq
\frac{(p-1)\nu}{|\nabla u_0(x)|^{2-p}}\,|\xi|^2 \geq
 \frac{(p-1)\nu}{\|\nabla u_0\|_{\infty}^{2-p}}\,|\xi|^2 \\
\qquad\text{for any $x\in\Omega\setminus Z_{u_0}$
and $\xi\in\R^N$}\,,
\end{multline*}
whence
\[
\frac{(p-1)\nu}{\|\nabla u_0\|_{\infty}^{2-p}}\,
\int_{\Omega} \varphi'(e_j)|\nabla e_j|^2\,dx
- \int_{\Omega} D_sg(x,u_0)e_j \varphi(e_j) \,dx
\leq 0\,.
\]
Since $D_sg(x,u_0)\in L^{\infty}(\Omega)$, it is standard
(see e.g.~\cite{ladyzhenskaya_uraltseva1968})
that $e_j\in L^{\infty}(\Omega)$, whence 
$V\subseteq X_{u_0}\cap L^{\infty}(\Omega)
\subseteq W^{1,p}_0(\Omega)$, as $p<2$.
\par
If we set
\[
\widetilde{W} = \left\{w\in L^1(\Omega):\,\,
\int_{\Omega} vw\,dx = 0\quad\text{for any $v\in V$}\right\}\,,
\]
then $\widetilde{W}$ is a closed linear subspace of $L^1(\Omega)$
and
\[
L^1(\Omega) = V\oplus \widetilde{W}\,.
\]
Since
\begin{alignat*}{3}
&\int_{\Omega}
\underline{\Psi}''(\nabla u_0)[\nabla u]^2\,dx
- \int_{\Omega} D_sg(x,u_0) u^2\,dx = Q_{u_0}(u)
&&\qquad\text{if $u\in X_{u_0}$}\,,\\
&\int_{\Omega}
\underline{\Psi}''(\nabla u_0)[\nabla u]^2\,dx
- \int_{\Omega} D_sg(x,u_0) u^2\,dx = +\infty
&&\qquad\text{if $u\in W^{1,2}_0(\Omega)\setminus X_{u_0}$}\,,
\end{alignat*}
the other assertions easily follow.
\par
In the case $\kappa>0$, one has $X_{u_0}=W^{1,2}_0(\Omega)$ 
and the adaptation of the previous argument is very simple
if $1<p\leq 2$.
If $p>2$, one has to remark that $\Psi''(\nabla u_0)$
is continuous.
By standard regularity results (see 
e.g.~\cite[Theorem~7.6]{simader1972}) it follows that
$e_j\in W^{1,p}_0(\Omega)$, whence 
$V\subseteq W^{1,p}_0(\Omega)$.
\end{proof}
In the following, we consider a direct sum decomposition as in
the previous proposition.
In particular, the projection
$\widetilde{P}_V:L^1(\Omega)\rightarrow V$, associated with the
direct sum decomposition, is continuous with respect to the
$L^1$-topology.
Since $V\subseteq W^{1,2}_0(\Omega)\cap L^{\infty}(\Omega)$ is
finite dimensional, it is equivalent to consider the norm
of $W^{1,2}_0(\Omega)\cap L^{\infty}(\Omega)$ on $V$.
\par
Then we set $W=\widetilde{W}\cap \w$, which is a closed linear
subspace of $\w$, so that
\[
W^{1,p}_0(\Omega) = V \oplus W
\]
and $P_V = \widetilde{P}_V\bigl|_{W^{1,p}_0}$ is 
$L^1$-continuous as well.
\par
We also set, for any $r>0$,
\begin{alignat*}{3}
& B_r && =\left\{u\in W^{1,p}_0(\Omega):\,\,
\|\nabla u\|_p < r\right\}\,,\\
& D_r && =\left\{u\in W^{1,p}_0(\Omega):\,\,
\|\nabla u\|_p\leq r\right\}\,.
\end{alignat*}
\begin{lem}
\label{lem:QconvW}
For any $M>0$, there exist $r,\delta>0$ such that, for every
$u \in (u_0+D_r)\cap W^{1, \infty}(\Omega)$ with
$\|u\|_{\infty}+ \|\nabla u\|_{\infty} \leq M$ and every
$w \in \widetilde{W} \cap W^{1,2}_0(\Omega)$, one has
\[
\int_{\Omega} \underline{\Psi}''(\nabla u)[\nabla w]^2\,dx
- \int_{\Omega} D_sg(x,u) w^2\,dx
\geq \delta \int_{\Omega} |\nabla w|^2\,dx \,.
\]
\end{lem}
\begin{proof}
Assume, for a contradiction, that there exist a sequence
$(v_k)$ in $\w\cap W^{1,\infty}(\Omega)$, strongly convergent to
$u_0$ in $W^{1,p}_0(\Omega)$ and bounded in $W^{1,\infty}(\Omega)$,
and a sequence $(w_k)$ in $\widetilde{W} \cap W^{1,2}_0(\Omega)$ 
such that
\begin{equation}
\label{eq:QconvW}
\int_{\Omega} \underline{\Psi}''(\nabla v_k)[\nabla w_k]^2\,dx
- \int_{\Omega} D_sg(x,v_k) w_k^2\,dx < \frac{1}{k} \,
\int_{\Omega} |\nabla w_k|^2\,dx \,.
\end{equation}
Without loss of generality, we may assume that
$\|\nabla w_k\|_2=1$.
Then, up to a subsequence, $(w_k)$ is
weakly convergent to some $w$ in $W^{1,2}_0(\Omega)$.
In particular, $w\in \widetilde{W}$.
From Theorem~\ref{thm:Qlsc} we infer that
\[
\int_{\Omega} \underline{\Psi}''(\nabla u_0)[\nabla w]^2\,dx
- \int_{\Omega} D_sg(x,u_0) w^2\,dx\leq 0\,,
\]
whence $w=0$.
\par
Coming back to~\eqref{eq:QconvW}, now we deduce that
\[
\lim_k
\int_{\Omega} \underline{\Psi}''(\nabla v_k)[\nabla w_k]^2\,dx =0\,.
\]
Since $(\nabla v_k)$ is bounded in $L^{\infty}(\Omega)$,
in both cases $\kappa=0$ with $1<p<2$ and $\kappa>0$
with $1<p<\infty$ we infer that $\nabla w_k\to 0$ in $L^2(\Omega)$.
Since $\|\nabla w_k\|_2=1$, a contradiction follows.
\end{proof}
\begin{thm}
\label{thm:split}
There exist $M, r>0$ and $\beta\in]0,1]$ such that:
\begin{itemize}
\item[$(a)$]
the map $f'$ is of class~$(S)_+$ on $u_0+D_{2r}$;
\item[$(b)$]
for every $v\in V\cap D_r$, the derivative of the functional
\[
\begin{array}{ccc}
W & \rightarrow & \R \\
w &\mapsto & f(u_0+v+w)
\end{array}
\]
is of class~$(S)_+$ on $W\cap D_r$;
moreover, if $w$ is a critical point of such
a functional with $w\in D_r$, then
$v+w\in C^{1,\beta}(\overline\Omega)$ and
\[
\|v+w\|_{C^{1,\beta}} < M\,;
\]
finally, the functional $\left\{w\mapsto f(u_0+v+w)\right\}$
is strictly convex on
\begin{align*}
& \hskip50pt \bigl\{w\in W\cap D_r:\,\,
\text{$(v+w)\in W^{1,\infty}(\Omega)$
}\bigr. \\
& \hskip200pt \bigl. \text{
and $\|v+w\|_{\infty} + \|\nabla(v+w)\|_{\infty}
\leq M$}\bigr\}\,;
\end{align*}
\item[$(c)$]
$u_0$ is a strict local minimum of $f$ along
$u_0+W$ for the $W^{1,p}_0(\Omega)$-topology.
\end{itemize}
\end{thm}
\begin{proof}
As already observed in the proof of Theorem~\ref{thm:locminW1p},
there exists $r>0$ such that the map $f'$ is of class~$(S)_+$ 
on $u_0+D_{2r}$.
It follows that $\{w \mapsto  f(u_0+v+w)\}$
is of class~$(S)_+$ on $W\cap D_r$
for any $v\in V\cap D_r$.
Moreover, if $w$ is a critical point, we have
\[
\langle f'(u_0+v+w),u-P_Vu\rangle = 0
\qquad\text{for any $u\in W^{1,p}_0(\Omega)$}\,,
\]
whence
\[
\langle f'(u_0+v+w),u\rangle =
\langle f'(u_0+v+w),P_Vu\rangle
\qquad\text{for any $u\in W^{1,p}_0(\Omega)$}\,.
\]
Since $P_V$ is continuous from the topology of
$L^1(\Omega)$ to that of $W^{1,p}_0(\Omega)$
and $f'$ is bounded on bounded sets, it follows that
\[
\langle f'(u_0+v+w),P_Vu\rangle = \int_{\Omega} z u\,dx
\qquad\text{for any $u\in W^{1,p}_0(\Omega)$}
\]
with $z$ uniformly bounded in $L^{\infty}(\Omega)$
with respect to $v$ and $w$, whence
\[
\langle f'(u_0+v+w),u\rangle = \int_{\Omega} z u\,dx
\qquad\text{for any $u\in W^{1,p}_0(\Omega)$}\,.
\]
From Theorems~\ref{thm:regLinfty} and~\ref{thm:regC1},
possibly by decreasing $r$, we conclude that
$u_0+v+w$, hence $v+w$, is uniformly bounded in
$C^{1,\beta}(\overline\Omega)$.
\par
Finally, again by decreasing $r$, we infer by
Lemma~\ref{lem:QconvW} that
\begin{equation}
\label{eq:delta}
\int_{\Omega} \underline{\Psi}''(\nabla(u_0+u))[\nabla w]^2\,dx
- \int_{\Omega} D_sg(x,u_0+u) w^2\,dx
\geq \delta \int_{\Omega} |\nabla w|^2\,dx
\end{equation}
for every
$u \in D_{2r}\cap W^{1, \infty}(\Omega)$ with
$\|u\|_{\infty}+ \|\nabla u\|_{\infty} \leq M$ and every
$w \in \widetilde{W} \cap W^{1,2}_0(\Omega)$.
\par
If $v\in V\cap D_r$, $t \in [0,1]$ and $w_0, w_1\in W\cap D_r$
with $(v+w_j)\in W^{1,\infty}(\Omega)$ and
\[
\|v+w_j\|_{\infty} + \|\nabla(v+w_j)\|_{\infty}
\leq M\,,
\]
we have $w_j\in W^{1,p}_0(\Omega)\cap W^{1,2}(\Omega)$, hence
$w_j \in W^{1,2}_0(\Omega)$, as $\partial\Omega$ is smooth enough.
By Proposition~\ref{prop:PsiTaylor} and~\eqref{eq:delta}
we easily deduce that
\begin{multline*}
(1-t)f(u_0+v+w_0)+t f(u_0+v+w_1) \\
\geq f(u_0+v+(1-t)w_0+tw_1)
+ \dfrac{\delta}{2}\,t(1-t)\,
\int_{\Omega} |\nabla w_1-\nabla w_0|^2\,dx\,.
\end{multline*}
Therefore $\left\{w\mapsto f(u_0+v+w)\right\}$
is strictly convex.
\par
In particular, the critical point $u_0$ is a strict local minimum 
of $f$ along $u_0+(W\cap C^1(\overline{\Omega}))$ for the 
$C^1(\overline{\Omega})$-topology.
From Theorem~\ref{thm:locminW1p} we infer that
$u_0$ is a strict local minimum of~$f$ along
$u_0+W$ for the $W^{1,p}_0(\Omega)$-topology.
\end{proof}
\begin{thm}
\label{thm:parmin}
There exist $M, r>0$, $\beta\in]0,1]$ and  $\varrho \in ]0,r]$
such that, for every $v \in V\cap D_{\varrho}$, there exists one
and only one $\overline  w \in W \cap D_r$ such that
\[
f(u_0 + v + \overline w)\leq f(u_0 + v + w)
\qquad\text{for any $w \in W \cap D_r$}\,.
\]
Moreover,
$v + \overline w \in C^{1,\beta}(\overline\Omega)$
with $\|v + \overline w\|_{C^{1,\beta}}\leq M$,
$\overline w\in B_r$ and $\overline w$ is the unique
critical point of $\{w\mapsto f(u_0+v+w)\}$ in
$W\cap D_r$.
\par
Finally, if we set $\psi(v)=\overline w$, the map
\[
\left\{v\mapsto v+\psi(v)\right\}
\]
is continuous from $V\cap D_{\varrho}$ into
$C^1(\overline\Omega)$, while the map $\psi$ is
continuous from $V\cap D_{\varrho}$ into
$W^{1,p}_0(\Omega)\cap L^{\infty}(\Omega)$.
Moreover, $\psi(0)=0$.
\par
In the case $\kappa>0$, the map $\psi$
is also of class $C^1$ from $V\cap B_{\varrho}$ into
$W^{1,2}_0(\Omega)$ and, for every $z\in V\cap B_{\varrho}$
and $v\in V$, we have that $\psi'(z)v$ is the minimum point of the
functional
\begin{multline*}
\biggl\{w\mapsto \dfrac{1}{2}\,
\int_{\Omega} \left\{\Psi''(\nabla u)[\nabla w]^2 
- D_sg(x,u) w^2\right\}\,dx \biggr. \\
\biggl. + \int_{\Omega} \left\{\Psi''(\nabla u)[\nabla v,\nabla w]
- D_sg(x,u) v w\right\}\,dx
\biggr\}
\end{multline*}
on $\widetilde{W}\cap W^{1,2}_0(\Omega)$, where $u=u_0+z+\psi(z)$.
Moreover, $\psi'(0)=0$.
\end{thm}
\begin{proof}
Let $M, r>0$ and $\beta\in]0,1]$ be as in Theorem~\ref{thm:split}.
In particular, we may suppose that
$f(u_0)<f(u_0+w)$ for every $w\in W\cap D_r$ with $w\neq 0$.
By Lemma~\ref{lem:QconvW}
we may also assume that there exists $\delta>0$ such that
\begin{equation}
\label{eq:QcoercW}
\int_{\Omega} \underline{\Psi}''(\nabla(u_0+u))[\nabla w]^2\,dx
- \int_{\Omega} D_sg(x,u_0+u) w^2\,dx
\geq \delta \int_{\Omega} |\nabla w|^2\,dx
\end{equation}
for every
$u \in D_{2r}\cap C^{1,\beta}(\overline\Omega)$
with  $\|u\|_{C^{1,\beta}} \leq M$ and every
$w \in \widetilde{W} \cap W^{1,2}_0(\Omega)$.
\par
We claim that there exists $\varrho\in]0,r]$
such that
\[
f(u_0+v)<f(u_0+v+w)
\qquad\text{for any $v\in V\cap D_{\varrho}$
and any $w\in W$ with $\|\nabla w\|_p = r$.}
\]
By contradiction, let $(v_k)$ be a sequence in $V$
with $v_k\to 0$ and let $(w_k)$ be a sequence in $W$
with $\|\nabla w_k\|_p = r$ and
$f(u_0+v_k)\geq f(u_0+v_k+w_k)$.
Up to a subsequence, $(w_k)$ is weakly convergent
to some $w\in W\cap D_r$.
Then $(u_0+v_k+w_k)$ is weakly convergent to $u_0+w$ with
\[
\limsup_k f(u_0+v_k+w_k) \leq
\lim_k f(u_0+v_k) = f(u_0) \leq f(u_0+w)\,.
\]
Combining Proposition~\ref{prop:S+} with Theorem~\ref{thm:split},
we deduce that $(u_0+v_k+w_k)$ is strongly convergent to $u_0+w$,
whence $f(u_0+w)=f(u_0)$ with $\|\nabla w\|_p = r$,
and a contradiction follows.
\par
Again from Proposition~\ref{prop:S+} and Theorem~\ref{thm:split}
we know that $\{w\mapsto f(u_0+v+w)\}$ is weakly lower
semicontinuous on $W\cap D_r$ for any $v\in V\cap D_{\varrho}$.
Therefore there exists a minimum point
$\overline w\in W\cap D_r$ and in fact
$\overline w \in B_r$.
In particular, we have
\[
\langle f'(u_0 + v + \overline w),w\rangle = 0
\qquad\text{for any $w\in W$}\,.
\]
From Theorem~\ref{thm:split} we infer that
$v + \overline w \in C^{1,\beta}(\overline\Omega)$
with $\|v + \overline w\|_{C^{1,\beta}}\leq M$.
Since $\left\{w\mapsto f(u_0+v+w)\right\}$ is strictly
convex on
\[
\left\{w\in W\cap D_r:\,\,
\text{$(v+w)\in W^{1,\infty}(\Omega)$
and $\|v+w\|_{\infty} + \|\nabla(v+w)\|_{\infty}
\leq M$}\right\}\,,
\]
the minimum is unique.
If $v=0$, then $\overline w=0$.
\par
Finally, if we set $\psi(v)=\overline w$, the map
$\left\{v\mapsto v+\psi(v)\right\}$
is defined from $V\cap D_{\varrho}$ into
\[
\left\{u\in C^{1,\beta}(\overline\Omega):\,\,
\|u\|_{C^{1,\beta}}\leq M\right\}\,,
\]
which is a compact subset of $C^1(\overline\Omega)$,
and has closed graph, as $f$ is continuous.
Therefore it is a continuous map.
The continuity of $\psi$ follows.
\par
In the case $\kappa>0$, the function $\Psi$ is of
class~$C^2$ on $\R^N$.
Therefore, there exists $C>0$ such that
\begin{equation}
\label{eq:Qcont}
\left|\int_{\Omega} \Psi''(\nabla(u_0+u))
[\nabla u_1,\nabla u_2]\,dx
- \int_{\Omega} D_sg(x,u_0+u) u_1 u_2\,dx\right|
\leq C \,\|\nabla u_1\|_2 \,\|\nabla u_2\|_2
\end{equation}
for every
$u \in D_{2r}\cap C^{1,\beta}(\overline\Omega)$
with  $\|u\|_{C^{1,\beta}} \leq M$ and every
$u_1, u_2 \in W^{1,2}_0(\Omega)$.
Moreover, we have
\begin{multline}
\label{eq:f'Lag}
\langle f'(u_0+v_1+\psi(v_1)),u\rangle
- \langle f'(u_0+v_0+\psi(v_0)),u\rangle \\
= \int_0^1 \int_{\Omega} \biggl\{\Psi''(\nabla \gamma_t)
[\nabla(v_1-v_0+\psi(v_1)-\psi(v_0)),\nabla u] \biggr. \\
- \biggl.
D_sg(x,\gamma_t) (v_1-v_0+\psi(v_1)-\psi(v_0)) u\biggr\}\,dx\,dt
\end{multline}
for any $v_0, v_1\in V\cap D_{\varrho}$ and
$u\in W^{1,2}_0(\Omega)$, where
\[
\gamma_t=u_0+v_0+\psi(v_0)+t(v_1-v_0+\psi(v_1)-\psi(v_0))\,.
\]
Now let $z\in V\cap B_{\varrho}$ and let $u=u_0+z+\psi(z)$.
From~\eqref{eq:QcoercW} and~\eqref{eq:Qcont} it follows that,
for every $v\in V$, the functional
\begin{multline*}
\biggl\{w\mapsto \dfrac{1}{2}\,
\int_{\Omega} \left\{\Psi''(\nabla u)[\nabla w]^2
- D_sg(x,u) w^2\right\}\,dx \biggr. \\
\biggl. + \int_{\Omega} \left\{\Psi''(\nabla u)[\nabla v,\nabla w]
- D_sg(x,u) v w\right\}\,dx
\biggr\}
\end{multline*}
admits one and only one minimum point $L_zv$ in
$\widetilde{W}\cap W^{1,2}_0(\Omega)$, which satisfies
\begin{multline}
\label{eq:psi'}
\int_{\Omega} \left\{\Psi''(\nabla u)[\nabla(L_zv),\nabla w]
- D_sg(x,u) (L_zv) w\right\}\,dx \\
= - \int_{\Omega} \left\{\Psi''(\nabla u)[\nabla v,\nabla w]
- D_sg(x,u) v w\right\}\,dx
\quad\text{for any $w\in \widetilde{W}\cap W^{1,2}_0(\Omega)$}\,.
\end{multline}
Moreover, the map $L_z:V\rightarrow W^{1,2}_0(\Omega)$ is linear and
continuous, as $V$ is finite dimensional.
Since $Q_{u_0}(v+w)=Q_{u_0}(v)+Q_{u_0}(w)$ for any $v\in V$ and
$w\in \widetilde{W}\cap W^{1,2}_0(\Omega)$, we also have $L_0=0$.
\par
By~\eqref{eq:f'Lag}, for every $v_0, v_1\in V\cap B_{\varrho}$ and
$w\in \widetilde{W}\cap W^{1,2}_0(\Omega)$, it holds
\begin{alignat*}{3}
&0 &&= \langle f'(u_0+v_1+\psi(v_1)),w\rangle
- \langle f'(u_0+v_0+\psi(v_0)),w\rangle \\
&&&
= \int_0^1 \int_{\Omega} \biggl\{\Psi''(\nabla \gamma_t)
[\nabla(v_1-v_0+\psi(v_1)-\psi(v_0)),\nabla w] \biggr. \\
&&& \qquad\qquad\qquad\qquad\qquad\qquad
- \biggl.
D_sg(x,\gamma_t) (v_1-v_0+\psi(v_1)-\psi(v_0)) w\biggr\}\,dx\,dt\,.
\end{alignat*}
Taking into account~\eqref{eq:psi'}, we deduce that
\[
\begin{split}
\int_0^1 \int_{\Omega} & \biggl\{ \Psi''(\nabla \gamma_t)
[\nabla(\psi(v_1)-\psi(v_0)),\nabla w]
- D_sg(x,\gamma_t) (\psi(v_1)-\psi(v_0)) w \biggr\}\,dx\,dt \\
& \qquad\qquad
- \int_{\Omega}
\biggl\{\Psi''(\nabla u)[\nabla(L_z(v_1-v_0)),\nabla w]
- D_sg(x,u) (L_z(v_1-v_0)) w\biggr\}\,dx \\
& =
- \int_0^1 \int_{\Omega} \biggl\{
\left[\Psi''(\nabla \gamma_t) - \Psi''(\nabla u)\right]
[\nabla(v_1-v_0),\nabla w] \\
& \qquad\qquad\qquad\qquad\qquad\qquad\qquad
- \left[D_sg(x,\gamma_t) - D_sg(x,u)\right](v_1-v_0) w\biggr\}
\,dx\,dt\,.
\end{split}
\]
It follows
\[
\begin{split}
\int_0^1 \int_{\Omega} & \biggl\{ \Psi''(\nabla \gamma_t)
[\nabla(\psi(v_1)-\psi(v_0)-L_z(v_1-v_0)),\nabla w] \\
& \qquad\qquad\qquad\qquad
- D_sg(x,\gamma_t) (\psi(v_1)-\psi(v_0)-L_z(v_1-v_0)) w
\biggr\}\,dx\,dt \\
& =
- \int_0^1 \int_{\Omega}\biggl\{
\left[\Psi''(\nabla \gamma_t) - \Psi''(\nabla u)\right]
[\nabla(L_z(v_1-v_0)),\nabla w] \\
& \qquad\qquad\qquad\qquad\qquad
- \left[D_sg(x,\gamma_t) - D_sg(x,u)\right]
(L_z(v_1-v_0)) w\biggr\}\,dx\,dt \\
& \qquad
- \int_0^1 \int_{\Omega} \biggl\{
\left[\Psi''(\nabla \gamma_t) - \Psi''(\nabla u)\right]
[\nabla (v_1-v_0),\nabla w] \\
& \qquad\qquad\qquad\qquad\qquad\qquad
- \left[
D_sg(x,\gamma_t) - D_sg(x,u)\right](v_1-v_0) w\biggr\}\,dx\,dt \,.
\end{split}
\]
Since the map $\left\{v\mapsto v+\psi(v)\right\}$
is continuous from $V\cap B_{\varrho}$ into
$C^1(\overline\Omega)$, from~\eqref{eq:QcoercW} we infer that
\[
\lim_{\substack{(v_0,v_1)\to (z,z)\\ v_0\neq v_1}}\,
\dfrac{\| \nabla(\psi(v_1)-\psi(v_0)-L_z(v_1-v_0))\|_2}{
\|\nabla(v_1-v_0)\|_2} = 0\,.
\]
Therefore $\psi$ is of class~$C^1$ from
$V\cap B_{\varrho}$ into $W^{1,2}_0(\Omega)$
and $\psi'(z) = L_z$.
\end{proof}


\section{The finite dimensional reduction}
\label{sect:red}
Throughout this section we keep the assumptions and the notations
of Section~\ref{sect:parmin}.
We also define the reduced functional
$\varphi:V\cap B_{\varrho}\rightarrow\R$ as
\[
\varphi(v) = f(u_0+v+\psi(v))
=\min\left\{f(u_0+v+w):\,\,w\in W\cap D_r\right\}\,.
\]
\begin{thm}
\label{thm:phiC1}
Let $\kappa>0$ with $1<p<\infty$
or $\kappa=0$ with $1<p<2$.
Then the functional $\varphi$ is of class $C^1$ and
\begin{equation}
\label{eq:phiC1}
\langle\varphi'(z),v\rangle =
\langle f'(u_0+z+\psi(z)),v\rangle
\qquad\text{for any $z\in V\cap B_{\varrho}$ and $v\in V$}\,.
\end{equation}
In particular, $0$ is a critical point of $\varphi$.
Moreover, we have
\[
C_m(\varphi,0) \approx C_m(f,u_0)
\qquad\text{for any $m \geq 0$}\,.
\]
Finally, $0$ is an isolated critical point of~$\varphi$ if and
only if $u_0$ is an isolated critical point of~$f$.
\end{thm}
\begin{proof}
For any $v_0, v_1\in V\cap B_{\varrho}$, we have
\begin{alignat*}{3}
&\varphi(v_1) &&= f(u_0+v_1+\psi(v_1)) \\
&&& = f(u_0+v_0+\psi(v_1)) +
\langle f'(u_0+v_0+t(v_1-v_0)+\psi(v_1)),v_1-v_0\rangle \\
&&& \geq f(u_0+v_0+\psi(v_0)) +
\langle f'(u_0+v_0+t(v_1-v_0)+\psi(v_1)),v_1-v_0\rangle \\
&&& = \varphi(v_0)
+ \langle f'(u_0+v_0+t(v_1-v_0)+\psi(v_1)),v_1-v_0\rangle
\end{alignat*}
for some $t\in]0,1[$.
Since $\psi$ is continuous from $V\cap B_{\varrho}$
into $W^{1,p}_0(\Omega)$, it follows that
\[
\liminf_{\substack{(v_0,v_1)\to (z,z)\\ v_0\neq v_1}}\,
\dfrac{\varphi(v_1)-\varphi(v_0)-
\langle f'(u_0+z+\psi(z)),v_1-v_0\rangle}{\|v_1-v_0\|} \geq 0\,.
\]
We also have
\begin{alignat*}{3}
&\varphi(v_1) &&= f(u_0+v_1+\psi(v_1))
\leq f(u_0+v_1+\psi(v_0)) \\
&&& = f(u_0+v_0+\psi(v_0)) +
\langle f'(u_0+v_0+t(v_1-v_0)+\psi(v_0)),v_1-v_0\rangle \\
&&& = \varphi(v_0)
+ \langle f'(u_0+v_0+t(v_1-v_0)+\psi(v_0)),v_1-v_0\rangle
\end{alignat*}
for some $t\in]0,1[$, whence
\[
\limsup_{\substack{(v_0,v_1)\to (z,z)\\ v_0\neq v_1}}\,
\dfrac{\varphi(v_1)-\varphi(v_0)-
\langle f'(u_0+z+\psi(z)),v_1-v_0\rangle}{\|v_1-v_0\|} \leq 0\,.
\]
Therefore $\varphi$ is of class $C^1$ with
\[
\langle\varphi'(z),v\rangle =
\langle f'(u_0+z+\psi(z)),v\rangle\,.
\]
Since $\psi(0)=0$, we also have $\varphi'(0)=0$.
\par
Now consider
\[
Y = \left\{u_0+ z + \psi(z):\,\,
z \in V \cap B_{\rho} \right\}
\]
endowed with the $W^{1,p}_0(\Omega)$-topology.
Since $\left\{z\mapsto u_0+ z + \psi(z)\right\}$
is a homeomorphism from $V \cap B_{\rho}$ onto $Y$
which sends $0$ into $u_0$, it is clear that
\[
C_m(\varphi,0) \approx C_m(f\bigl|_Y,u_0)
\qquad\text{for any $m \geq 0$}\,.
\]
On the other hand, from
Proposition~\ref{prop:S+} and Theorem~\ref{thm:split}
we see that the functional $\left\{w\mapsto f(u_0+v+w)\right\}$
satisfies the Palais-Smale condition over $W\cap D_r$
for any $v\in V\cap B_{\varrho}$.
Moreover, $\psi(v)$ is the unique critical point, in fact
the minimum, of such a functional in $W\cap D_r$.
Arguing as in the Second Deformation Lemma, it is possible to
define a deformation
\[
\mathcal{H}:
\left(u_0+(V\cap B_{\varrho})+(W\cap D_r)\right)
\times[0,1] \rightarrow
\left(u_0+(V\cap B_{\varrho})+(W\cap D_r)\right)
\]
such that
\begin{align*}
& \mathcal{H}(u,t) - u \in W\,,\qquad
f(\mathcal{H}(u,t)) \leq f(u)\,,\\
&\mathcal{H}(u,1) \in Y\,,\qquad
\mathcal{H}(u,t)=u \qquad\text{if $u\in Y$}\,,
\end{align*}
whence
\[
H^m(f^c,f^c\setminus\{u_0\}) \approx
H^m(f^c\cap Y,(f^c\cap Y)\setminus\{u_0\})\,.
\]
This is proved in~\cite[Theorem 5.4]{cingolani_vannella2003-aihp}
in the case $p>2$, but the argument works also for
$1<p\leq 2$.
See also~\cite[Theorem~4.7]{lancelotti2002} in a nonsmooth setting.
\par
Therefore we have
\[
C_m(\varphi,0) \approx C_m(f\bigl|_Y,u_0)
\approx C_m(f,u_0)
\qquad\text{for any $m \geq 0$}\,.
\]
Since any critical point $u$ of $f$ in
$u_0+(V\cap B_{\varrho})+(W\cap D_r)$ must be of
the form $u=u_0+z+\psi(z)$ with $z\in V\cap B_{\varrho}$,
from~\eqref{eq:phiC1} we infer that $0$ is isolated for
$\varphi$ if and only if $u_0$ is isolated for $f$.
\end{proof}
\begin{thm}
\label{thm:phiC2}
Let $\kappa>0$ with $1<p<\infty$.
Then $\varphi$ is of class $C^2$ and
\begin{align}
\label{eq:phiC2}
\varphi''(z)[v]^2 &=
\int_{\Omega}\biggl\{\Psi''(\nabla u)[\nabla (v+\psi'(z)v)]^2
- D_sg(x,u)(v+\psi'(z)v)^2\biggr\}\,dx \\
\nonumber &
= \int_{\Omega}\biggl\{\Psi''(\nabla u)[\nabla v]^2
- D_sg(x,u)v^2\biggr\}\,dx \\
\nonumber &\qquad\qquad
- \int_{\Omega}\biggl\{\Psi''(\nabla u)[\nabla(\psi'(z)v)]^2
- D_sg(x,u)(\psi'(z)v)^2\biggr\}\,dx \\
\nonumber &\qquad\qquad\qquad
\text{for any $z\in V\cap B_{\varrho}$ and $v\in V$,
where $u=u_0+z+\psi(z)$}\,.
\end{align}
In particular, we have
\[
\varphi''(0)[v]^2 =
\int_{\Omega}\biggl\{\Psi''(\nabla u_0)[\nabla v]^2
- D_sg(x,u_0)v^2\biggr\}\,dx
\qquad\text{for any $v\in V$}\,.
\]
\end{thm}
\begin{proof}
By Theorem~\ref{thm:parmin}, the map $\psi$ is of class~$C^1$ from
$V\cap B_{\varrho}$ into $W^{1,2}_0(\Omega)$ with $\psi(0)=0$ and
$\psi'(0)=0$.
For any $z\in V\cap B_{\varrho}$, let $L_z:V \rightarrow V'$
be the linear map defined by
\begin{multline}
\label{eq:psi''}
\langle L_zv_1,v_2\rangle =
\int_{\Omega} \biggl\{\Psi''(\nabla u)
[\nabla v_1,\nabla v_2]
- D_sg(x,u) v_1 v_2\biggr\}\,dx\\
+ \int_{\Omega} \biggl\{\Psi''(\nabla u)
[\nabla(\psi'(z)v_1),\nabla v_2]
- D_sg(x,u) (\psi'(z)v_1) v_2\biggr\}\,dx \,,
\end{multline}
where $u=u_0+z+\psi(z)$.
By~\eqref{eq:f'Lag}, for every $v_0, v_1\in V\cap B_{\varrho}$ and
$v\in V$, we have
\begin{alignat*}{3}
&\langle\varphi'(v_1),v\rangle - \langle\varphi'(v_0),v\rangle &&=
\langle f'(u_0+v_1+\psi(v_1)),v\rangle
- \langle f'(u_0+v_0+\psi(v_0)),v\rangle \\
&&& = \int_0^1 \int_{\Omega} \biggl\{\Psi''(\nabla \gamma_t)
[\nabla(v_1-v_0+\psi(v_1)-\psi(v_0)),\nabla v] \biggr. \\
&&& \qquad\qquad\qquad
- \biggl.
D_sg(x,\gamma_t) (v_1-v_0+\psi(v_1)-\psi(v_0)) v\biggr\}\,dx\,dt\,,
\end{alignat*}
where $\gamma_t=u_0+v_0+\psi(v_0)+t(v_1-v_0+\psi(v_1)-\psi(v_0))$.
Taking into account~\eqref{eq:psi''}, we deduce that
\[
\begin{split}
\langle\varphi'(v_1),v\rangle & - \langle\varphi'(v_0),v\rangle
- \langle L_z(v_1-v_0),v\rangle \\
& =
\int_0^1 \int_{\Omega} \biggl\{\Psi''(\nabla \gamma_t)
[\nabla(v_1-v_0+\psi(v_1)-\psi(v_0)),\nabla v] \biggr. \\
& \qquad\qquad\qquad\qquad\qquad\qquad
- \biggl.
D_sg(x,\gamma_t) (v_1-v_0+\psi(v_1)-\psi(v_0)) v\biggr\}\,dx\,dt \\
& \qquad
- \int_{\Omega} \biggl\{\Psi''(\nabla u)
[\nabla (v_1-v_0),\nabla v]
- D_sg(x,u) (v_1-v_0) v\biggr\}\,dx  \\
& \qquad
- \int_{\Omega} \biggl\{\Psi''(\nabla u)
[\nabla(\psi'(z)(v_1-v_0)),\nabla v]
- D_sg(x,u) (\psi'(z)(v_1-v_0)) v\biggr\}\,dx \,.
\end{split}
\]
It follows
\[
\begin{split}
\langle\varphi'(v_1),v\rangle & - \langle\varphi'(v_0),v\rangle
- \langle L_z(v_1-v_0),v\rangle \\
& =
\int_0^1 \int_{\Omega} \biggl\{
\left[\Psi''(\nabla \gamma_t) - \Psi''(u)\right]
[\nabla(v_1-v_0),\nabla v] \biggr.  \\
& \qquad\qquad\qquad
- \biggl.
\left[D_sg(x,\gamma_t) - D_sg(x,u)\right]
(v_1-v_0) v\biggr\}\,dx\,dt \\
& \qquad
+ \int_0^1 \int_{\Omega} \biggl\{\Psi''(\nabla \gamma_t)
[\nabla(\psi(v_1)-\psi(v_0)-\psi'(z)(v_1-v_0)),\nabla v] \biggr. \\
& \qquad\qquad\qquad\qquad
- \biggl.
D_sg(x,\gamma_t) (\psi(v_1)-\psi(v_0)-\psi'(z)(v_1-v_0)) v
\biggr\}\,dx\,dt \\
& \qquad
+ \int_0^1 \int_{\Omega} \biggl\{
\left[\Psi''(\nabla \gamma_t) - \Psi''(u)\right]
[\nabla(\psi'(z)(v_1-v_0)),\nabla v] \biggr.  \\
& \qquad\qquad\qquad\qquad
- \biggl.
\left[D_sg(x,\gamma_t) - D_sg(x,u)\right]
(\psi'(z)(v_1-v_0)) v
\biggr\}\,dx\,dt \,.
\end{split}
\]
Since the map $\left\{v\mapsto v+\psi(v)\right\}$
is continuous from $V\cap B_{\varrho}$ into
$C^1(\overline\Omega)$, we infer that
\[
\lim_{\substack{(v_0,v_1)\to (z,z)\\ v_0\neq v_1}}\,
\dfrac{\langle\varphi'(v_1),v\rangle - \langle\varphi'(v_0),v\rangle
- \langle L_z(v_1-v_0),v\rangle}{
\|\nabla(v_1-v_0)\|_2} = 0
\qquad\text{for any $v\in V$}\,.
\]
Therefore $\varphi$ is of class~$C^2$ and
\begin{multline*}
\varphi''(z)[v_1,v_2] =
\int_{\Omega} \biggl\{\Psi''(\nabla u)
[\nabla v_1,\nabla v_2]
- D_sg(x,u) v_1 v_2\biggr\}\,dx \\
+ \int_{\Omega} \biggl\{\Psi''(\nabla u)
[\nabla(\psi'(z)v_1),\nabla v_2)]
- D_sg(x,u) (\psi'(z)v_1) v_2\biggr\}\,dx \,.
\end{multline*}
By Theorem~\ref{thm:parmin}, we also have
\begin{multline*}
\int_{\Omega} \left\{\Psi''(\nabla u)[\nabla (\psi'(z)v),\nabla w]
- D_sg(x,u) (\psi'(z)v) w \right\}\,dx \\
= - \int_{\Omega} \left\{\Psi''(\nabla u)[\nabla v,\nabla w]
- D_sg(x,u) v w  \right\}\,dx \\
\qquad\text{for any $v\in V$ and 
$w\in \widetilde{W}\cap W^{1,2}_0(\Omega)$}\,,
\end{multline*}
whence~\eqref{eq:phiC2}.
\par
Since $\psi(0)=0$ and $\psi'(0)=0$, the formula for
$\varphi''(0)$ follows.
\end{proof}
%


\section{Proof of the results of Section~\ref{sect:critical groups}}
\label{sect:proofcrit}
\noindent
\emph{Proof of Theorems~\ref{thm:zero}, \ref{thm:u=0}, 
\ref{thm:positivo}, \ref{thm:cornerstone} and~\ref{thm:positivo2}.}
\par\noindent
From Theorem~\ref{thm:phiC1} we know that
\[
C_m(f,u_0) \approx C_m(\varphi,0)
\qquad\text{for any $m \geq 0$}\,.
\]
Since the critical groups are defined using
Alexander-Spanier cohomology, it is clear that
$C_m(\varphi,0)=\{0\}$ whenever $m > \dim V = m^*(f,u_0)$,
both in the case $\kappa>0$ with $1<p<\infty$
and in the case $\kappa=0$ with $1<p<2$.
\par
In the particular case $u_0=0$ with $\kappa=0$ and $1<p<2$,
we clearly have $Z_{u_0}=\Omega$ and $X_{u_0}=\{0\}$, 
whence $m(f,0)=m^*(f,0)=0$, $V=\{0\}$
and $W=W^{1,p}_0(\Omega)$.
From Theorem~\ref{thm:parmin} it follows that $0$ is a
strict local minimum and an isolated critical point of~$f$.
By the excision property, it follows
\[
C_m(f,0) \approx
H^m \left(\{0\},\emptyset\right) \,,
\]
whence
\[
\left\{
\begin{array}{ll}
C_m(f,0) \approx \G
&\text{if $m = 0$}\,,\\
\noalign{\medskip}
C_m(f, 0) = \{0\}
&\text{if $m \neq 0$}\,.
\end{array}
\right.
\]
Now assume that $\kappa>0$ with $1<p<\infty$.
From Theorem~\ref{thm:phiC2} and
Proposition~\ref{prop:decomposition}
we infer that $\varphi$ is of class~$C^2$ with
\[
\varphi''(0)[v]^2 = Q_{u_0}(v) \leq 0
\qquad\text{for any $v\in V$}\,.
\]
Let $V_-$ be a subspace of $X_{u_0}=W^{1,2}_0(\Omega)$
of dimension $m(f,u_0)$ such that $Q_{u_0}$ is
negative definite on $V_-$.
Then it is easily seen that $Q_{u_0}$ is negative
definite also on $P_V(V_-)$, which has the same dimension
of $V_-$.
Therefore we may assume, without loss of generality, that
$V_- \subseteq V$ and we have
\[
\varphi''(0)[v]^2 = Q_{u_0}(v) < 0
\qquad\text{for any $v\in V_-\setminus \{0\}$}\,.
\]
It follows (see e.g.~\cite[Theorem~3.1]{lancelotti2002})
that $C_m(\varphi,0)=\{0\}$ whenever
$m<\dim V_-=m(f,u_0)$.
The proof of Theorems~\ref{thm:zero}, \ref{thm:u=0} 
and~\ref{thm:positivo} is complete.
\par
If $m(f,u_0)=m^*(f,u_0)$, we have $V_-=V$.
Then $0$ is a nondegenerate critical point of~$\varphi$
with Morse index $\dim V=m(f,u_0)$.
It follows that $0$ is an isolated critical point of
$\varphi$ and
\[
C_m(f,u_0) \approx C_m(\varphi,0) \approx
\delta_{m, m(f,u_0)} \G\,.
\]
Moreover, $u_0$ is an isolated critical point of $f$
by Theorem~\ref{thm:phiC1} and Theorem~\ref{thm:cornerstone}
follows.
\par
Finally, assume that $u_0$ is an isolated critical point
of~$f$ with $m(f,u_0)<m^*(f,u_0)$.
By Theorem~\ref{thm:phiC1} we infer that $0$ is an isolated
critical point of $\varphi$ and Theorem~\ref{thm:positivo2}
follows from \cite[Corollary~8.4]{mawhin_willem1989}.
\qed
\par\bigskip
\noindent
\emph{Proof of Theorem~\ref{thm:general}.}
\par\noindent
By Theorem~\ref{thm:positivo}, Remark~\ref{rem:p2} 
and Theorem~\ref{thm:u=0}, we have 
only to treat the case $\kappa=0$ with $p > 2$, so that
\[
Q_{u_0}(v) = Q_0(v) = - \int_\Omega g'(0)v^2\,dx
\qquad\forall v\in W^{1,2}_0(\Omega)\,.
\]
If $g'(0)=0$, we have $m(f,0)=0$, $m^*(f,0)=+\infty$
and the assertion is obvious.
\par
If $g'(0)<0$, we have $m(f,0)=m^*(f,0)=0$.
On the other hand, it is easily seen that
\[
f:W^{1,p}_0(\Omega)\cap C^1(\overline{\Omega})\rightarrow\R
\]
is strictly convex in a neighborhood of $0$ for the
$C^1(\overline{\Omega})$-topology.
In particular, $0$ is a strict local minimum for the
$C^1(\overline{\Omega})$-topology.
From Theorem~\ref{thm:locminW1p} we infer that $0$ is a strict 
local minimum of
\[
f:W^{1,p}_0(\Omega)\rightarrow\R
\]
for the $W^{1,p}_0(\Omega)$-topology.
By the excision property we have
\[
C_m(f,0) \approx
H^m \left(\{0\},\emptyset\right) 
\]
and the assertion follows.
\par
If $g'(0)>0$, we have $m(f,0)=m^*(f,0)=+\infty$.
If $p>N$, the functional $f$ is of class~$C^2$ 
on $W^{1,p}_0(\Omega)$ with
\[
f''(0)(v)^2 = - \int_\Omega g'(0)v^2\,dx
\qquad\forall v\in W^{1,p}_0(\Omega)\,.
\]
From \cite[Theorem~3.1]{lancelotti2002} we infer that
$C_m(f,0)=\{0\}$ for any $m$ and the assertion follows.
\par
If $p\leq N$, recall that
\[
|g(s)| \leq  C( 1+ |s|^{q})
\]
with $q<p^*-1$ if $p<N$, and consider a
$C^\infty$-function $\vartheta:\R\rightarrow[0,1]$ 
with $\vartheta(s)=1$ for $|s|\leq 1$
and $\vartheta(s)=0$ for $|s|\geq 2$.
Then define, for every $t\in[0,1]$, a $C^1$-functional 
$f_t:W^{1,p}_0(\Omega)\rightarrow\R$ as
\[
f_t(u) = 
\int_{\Omega} \Psi(\nabla u)\,dx
- \int_{\Omega} G_t(u)\,dx \,,
\]
where
\[
g_t(s) = g(\vartheta(ts)s)\,,\qquad
G_t(s) = \int_0^s g_t(\sigma)\,d\sigma\,.
\]
For any $t\in]0,1]$ the functional $f_t$ is of 
class $C^2$ with
\[
f_t''(0)(v)^2 = - \int_\Omega g'(0)v^2\,dx
\qquad\forall v\in W^{1,p}_0(\Omega)\,.
\]
Again from \cite[Theorem~3.1]{lancelotti2002} we infer that
$C_m(f_t,0)=\{0\}$ for any $t\in]0,1]$ and any $m$.
\par
Let $r>0$ be such that $0$ is the unique critical point
of $f_0=f$ in
\[
D_r =\left\{u\in W^{1,p}_0(\Omega):\,\,
\|\nabla u\|_p\leq r\right\}
\]
and such that the assertion of
Theorem~\ref{thm:regLinfty} holds for 
\[
\hat{g}(s) = C|s|^{q-1}s\,.
\]
Then the map $\left\{t\mapsto f_t\right\}$ is continuous
from $[0,1]$ into $C^1(D_r)$.
Moreover from~\cite[Theorem~3.5]{almi_degiovanni2013}
we infer that $f_t'$ is of class~$(S)_+$,
so that $f_t$ satisfies the Palais-Smale condition
over $D_r$, for any $t\in [0,1]$.
\par
We claim that there exists $\overline{t}\in]0,1]$
such that $0$ is the unique critical point of $f_t$
in $D_r$ whenever $0\leq t\leq\overline{t}$.
Assume, for a contradiction, that $t_k\to 0$
and $u_k\in D_r\setminus\{0\}$ is a critical point 
of $f_{t_k}$.
Then, for every $v\in W^{1,p}_0(\Omega)$ with
$v u_k\geq 0$, we have
\[
\begin{split}
\int_\Omega \nabla\Psi(\nabla u_k)\cdot\nabla v\,dx &=
\int_\Omega g_{t_k}(u_k)v\,dx\\
&\leq
\int_{\{u_k\neq 0\}} |g(\vartheta(t_ku_k)u_k)|\,|v|\,dx\\
&=
\int_{\{u_k\neq 0\}} \frac{|g(\vartheta(t_ku_k)u_k)|}{|u_k|}
\,u_kv\,dx\\
&\leq
\int_{\{u_k\neq 0\}} \frac{C(1+|u_k|^q)}{|u_k|}
\,u_kv\,dx \\
&=
\int_{\{u_k\neq 0\}} C\frac{u_k}{|u_k|}\,v\,dx 
+ \int_{\Omega} C|u_k|^{q-1}u_kv\,dx \,.
\end{split}
\]
It follows
\[
\int_\Omega \left[
\nabla\Psi(\nabla u_k)\cdot\nabla v - \hat{g}(u_k)v\right]\,dx 
\leq \langle \hat{w}_k,v\rangle \,,
\]
where
\[
\hat{w}_k=\left\{
\begin{array}{ll}
\displaystyle{
C\frac{u_k}{|u_k|}} &\text{where $u_k\neq 0$}\,,\\
0 &\text{where $u_k=0$}\,.
\end{array}
\right. 
\]
From Theorem~\ref{thm:regLinfty} we infer that
$(u_k)$ is bounded in $L^\infty(\Omega)$,
so that $\vartheta(t_ku_k)=1$ eventually as $k\to\infty$.
Then $u_k$ is a critical point of $f$ and a contradiction
follows.
\par
From~\cite[Theorem~5.2]{corvellec_hantoute2002} we deduce
that $C_m(f,0)\approx C_m(f_{\overline{t}},0)$
(for related results, see also
\cite[Theorem~I.5.6]{chang1993},
\cite[Theorem~3.1]{cingolani_degiovanni2009}
and \cite[Theorem~8.8]{mawhin_willem1989})
and the assertion follows.
\qed


\section{Proof of the main results}
\label{sect:proofs}
In this last section we prove the main results stated in the 
Introduction.
Let us recall some variants of the results 
of~\cite{cingolani_degiovanni2005} suited for our purposes.
We start with a saddle theorem, where linear subspaces are
substituted by symmetric cones.
\begin{thm}
\label{thm:saddle}
Let $X$ be a real Banach space and let $X_-, X_+$ be two
symmetric cones in~$X$ such that $X_+$ is closed in $X$,
$X_- \cap X_+ = \{0\}$ and such that
\[
\mathrm{Index}(X_-\setminus\{0\}) = 
\mathrm{Index}(X\setminus X_+) < +\infty\,.
\]
Let $r>0$ and let
\[
D_-=\left\{u\in X_-:\,\|u\| \leq r\right\} \,, \qquad
S_-=\left\{u\in X_-:\,\|u\| = r\right\} \,.
\]
Let $f:X\rightarrow\R$ be a function of class $C^1$ such that
\begin{gather*}
\inf_{X_+} f>-\infty\,,\qquad
\sup_{D_-} f < +\infty\,,\\
\text{if $X_-\neq \{0\}$, we have 
$f(u) < \inf_{X_+} f$ whenever $u\in S_-$}\,.
\end{gather*}
Set
\[
a= \inf_{X_+} f\,,\qquad
b= \sup_{D_-} f \,,\qquad
m= \mathrm{Index}(X_-\setminus\{0\})
\]
and assume that every sequence $(u_n)$ in $X$, with
\[
\text{$f(u_n)\to c\in[a,b]$ and 
$(1+\|u_n\|)\|f'(u_n)\|\to 0$}\,,
\]
admits a convergent subsequence 
(\emph{Cerami-Palais-Smale condition}) and that 
$f^{-1}([a,b])$ contains a finite
number of critical points.
\par
Then there exists a critical point $u$ of $f$ with
$a\leq f(u)\leq b$ and $C_m(f,u)\neq \{0\}$.
\end{thm}
\begin{proof}
From~\cite[Theorems~2.7 and~2.8]{degiovanni_lancelotti2007}
we infer that $(D_-,S_-)$ links $X_+$ cohomologically
in dimension $m$ over $\Z_2$.
According 
to~\cite[Remark~4.4]{corvellec1999},
the Cerami-Palais-Smale condition is just the 
usual Palais-Smale condition with respect to an
auxiliary distance function.
Then the assertion follows from~\cite[Theorem~5.2, 
Remark~5.3 and Theorem~7.5]{degiovanni2011}.
\end{proof}
\begin{thm}
\label{thm:cones}
Let $(\lambda_m)$ be defined as in
the Introduction and let $m\geq 0$ be such that
$\lambda_m<\lambda_{m+1}$.
If we set
\[
\begin{array}{ll}
\left\{
\begin{array}{lll}
X_- &= 
&\displaystyle{\left\{u\in W^{1,p}_0(\Omega):\,\,
\int_\Omega |\nabla u|^p\,dx \leq 
\lambda_m \int_\Omega |u|^p\,dx\right\}} \\
\noalign{\medskip}
X_+ &= 
&\displaystyle{\left\{u\in W^{1,p}_0(\Omega):\,\,
\int_\Omega |\nabla u|^p\,dx \geq 
\lambda_{m+1} \int_\Omega |u|^p\,dx\right\}}
\end{array}
\right.
&\qquad\text{if $m\geq 1$}\,,\\
\noalign{\medskip}
\left\{
\begin{array}{lll}
X_- &= 
&\displaystyle{\left\{0\right\}} \\
\noalign{\medskip}
X_+ &= 
& W^{1,p}_0(\Omega)
\end{array}
\right.
&\qquad\text{if $m=0$}\,,
\end{array}
\]
then $X_-, X_+$ are two closed symmetric cones 
in $W^{1,p}_0(\Omega)$ such that $X_- \cap X_+ = \{0\}$ 
and such that
\[
\mathrm{Index}(X_-\setminus\{0\}) = 
\mathrm{Index}(W^{1,p}_0(\Omega)\setminus X_+) = m \,.
\]
\end{thm}
\begin{proof}
If $m\geq 1$, the result is contained 
in~\cite[Theorem~3.2]{degiovanni_lancelotti2007}.
The case $m=0$ is obvious.
\end{proof}
Now let $f:W^{1,p}_0(\Omega)\to\R$ be the $C^1$-functional 
defined in~\eqref{eq:fmod} by setting
\[
f(u) = \int_{\Omega} \Psi_{p,\kappa}(\nabla u)\,dx
- \int_{\Omega} G(u)\,dx \,.
\]
\par\bigskip
\noindent
\emph{Proof of Theorem~\ref{thm:nonres}.}
\par\noindent
Let us show that $f$ satisfies the Cerami-Palais-Smale condition.
Let $(u_n)$ be a sequence in $W^{1,p}_0(\Omega)$ with
$f(u_n)$ bounded and $(1+\|u_n\|)\|f'(u_n)\| \to 0$,
so that
\begin{equation}
\label{eq:CPS}
\lim_n \,\langle f'(u_n),v-u_n\rangle = 0
\qquad\forall v\in W^{1,p}_0(\Omega)\,.
\end{equation}
First of all, let us show that $(u_n)$ is bounded in
$W^{1,p}_0(\Omega)$.
By contradiction, assume that $\|u_n\| \to \infty$
and set $z_n = \frac{u_n}{\|u_n\|}$.
Up to a subsequence, $z_n$ is convergent to some
$z$ weakly in $W^{1,p}_0(\Omega)$, strongly in $L^p(\Omega)$
and a.e. in~$\Omega$.
Since $\langle f'(u_n),z-z_n\rangle \to 0$,
dividing by $\|u_n \|^{p-1}$ and taking into
account~$(a)$, we get
\[
\lim_n \int_{\Omega} \left(\frac{\kappa^2}{\|u_n\|^2}+
|\nabla z_n|^2\right)^{\frac{p-2}{2}} \nabla z_n \cdot
\nabla(z - z_n) \, dx = 0 \,.
\]
By the convexity of $\Psi_{p,\kappa}$, it follows
\begin{multline*}
\limsup_n \int_\Omega |\nabla z_n|^p\,dx
\leq
\limsup_n \int_\Omega
\left(\frac{\kappa^2}{\|u_n\|^2}
+ |\nabla z_n|^2\right)^{\frac{p}{2}}\,dx \\
\leq
\lim_n \int_\Omega
\left(\frac{\kappa^2}{\|u_n\|^2}
+ |\nabla z|^2\right)^{\frac{p}{2}}\,dx =
\int_\Omega |\nabla z|^p\,dx \,,
\end{multline*}
so that $z_n \to z$ strongly in $W^{1,p}_0(\Omega)$
and $z\neq 0$.
\par
Given $v \in W^{1,p}_0(\Omega)$, we also have
$\langle f'(u_n),v\rangle \to 0$ whence,
dividing again by $\|u_n \|^{p-1}$,
\[
\lim_n \int_{\Omega} \left[\left(\frac{\kappa^2}{\|u_n\|^2}+
|\nabla z_n|^2\right)^{\frac{p-2}{2}} \nabla z_n \cdot
\nabla v 
- \frac{g(\|u_n\| z_n)}{\|u_n\|^{p-1}}\,v\right]\, dx = 0 \,.
\]
Taking again into account~$(a)$, we get
\[
\int_{\Omega} |\nabla z|^{p-2} \nabla z \cdot \nabla v \, dx =
\lambda \, \int_{\Omega} |z|^{p-2} z v \,dx 
\qquad\forall v \in W^{1,p}_0(\Omega)\,,
\]
which contradicts the assumption that
$\lambda \not\in\sigma(-\Delta_p)$.
Therefore $(u_n)$ is bounded in $W^{1,p}_0(\Omega)$, hence
convergent, up to a subsequence, to some
$u$ weakly in $W^{1,p}_0(\Omega)$.
\par
According 
to~\cite[Theorem~3.5]{almi_degiovanni2013}, the operator $f'$
is of class~$(S)_+$.
From~\eqref{eq:CPS} we infer that $(u_n)$ is strongly
convergent to $u$ in $W^{1,p}_0(\Omega)$.
\par
Now define $X_-, X_+$ according to Theorem~\ref{thm:cones}
with $m=m_\infty$, so that $X_-, X_+$ are 
two symmetric cones in $W^{1,p}_0(\Omega)$ satisfying
the assumptions of Theorem~\ref{thm:saddle} with
$\mathrm{Index}(X_-\setminus\{0\}) = m_\infty$.
Let us treat the case $m_\infty\geq 1$.
The case $m_\infty= 0$ is similar and simpler.
If
\[
\lambda_{m_\infty} 
< \alpha' < \alpha'' 
< \lambda 
< \beta' < \beta''  
< \lambda_{m_\infty +1}\,,
\]
taking into account assumption~$(a)$ we infer that there 
exists $C>0$ such that
\begin{alignat*}{7}
&\frac{\beta''}{p\lambda_{m_\infty +1}}
\,|\xi|^p - C &&\leq \Psi_{p,\kappa}(\xi) 
&&\leq \frac{\alpha'}{p\lambda_{m_\infty}}\,|\xi|^p +C
&&\qquad\forall \xi\in\R^N\,,\\
&\frac{\alpha''}{p}\,|s|^p - C &&\leq G(s) 
&&\leq \frac{\beta'}{p}\,|s|^p +C
&&\qquad\forall s\in\R\,.
\end{alignat*}
It easily follows that
\[
\inf_{X_+} f >-\infty\,,\qquad
\lim_{\substack{\|u\|\to\infty \\ u\in X_-}}\,f(u) = -\infty\,.
\]
In particular, there exists $r>0$ such that
\[
\forall u \in S_-:\,\,
f(u) < \inf_{X_+} f
\]
and, since $f$ is bounded on bounded subsets, we also have
$\sup\limits_{D_-} f < +\infty$.
\par
If $f$ has infinitely many critical points, we are done.
Otherwise, from Theorem~\ref{thm:saddle} we infer that
there exists a critical point $u$ of $f$ with
$C_{m_\infty}(f,u)\neq \{0\}$.
\par
Since $m_{\infty} \not\in [m(f,0),m^*(f,0)]$,
from Theorem~\ref{thm:general} we deduce that
$C_{m_\infty}(f,0) = \{0\}$.
Therefore $u\neq 0$ and the assertion follows.
\qed
\par\bigskip
In order to prove Theorem~\ref{thm:res}, we need an auxiliary
result.
\begin{prop}
\label{prop:res}
Let $\gamma\in\R$ and $\Gamma:\R\rightarrow\R$ be a function 
of class~$C^1$ such that
\begin{align*}
&\lim_{|s|\to\infty}\, \frac{\Gamma(s)}{|s|^p} = \gamma\,,\\
&\lim_{|s|\to\infty} \,
\left[p\Gamma(s) - s \Gamma'(s) \right]= +\infty \,.
\end{align*}
Then we have
\[
\lim_{|s|\to\infty}\, 
\left[\Gamma(s) - \gamma |s|^p \right]= +\infty \,.
\]
\end{prop}
\begin{proof}
Let $H(s)= \Gamma(s) -\gamma |s|^p$, so that
\begin{align*}
&\lim_{|s|\to\infty}\, \frac{H(s)}{|s|^p} = 0\,,\\
&\lim_{|s|\to\infty} \,
\left[p H(s) - s H'(s) \right]= +\infty \,.
\end{align*}
For every $M >0$, there exists $\overline{s} >0$ such that 
$p H(s) - s H'(s) \geq pM$ for any $s \geq \overline{s}$.
It follows 
\[
\biggl(\frac{H(s) - M}{s^p}\biggr)^\prime = 
\frac{s H'(s) - p H(s) + pM}{s^{p+1}} \leq 0
\qquad\forall s \geq \overline{s}\,,
\]
which implies that
\[
\frac{H(t)}{t^p}  -\frac{M}{t^{p}} 
\leq \frac{H(s)}{s^{p}}  - \frac{M}{s^{p}}
\qquad\text{whenever $t \geq s \geq \overline{s}$}\,.
\]
Passing to the limit as $t \to + \infty$, we get
\[
0 \leq \frac{H(s)}{s^p} - \frac{M}{s^{p}}
\qquad\forall s \geq \overline{s}\,,
\]
namely
\[
H(s) \geq M
\qquad\forall s \geq \overline{s}\,.
\]
Therefore
\[
\lim_{s\to+\infty}\,H(s) = +\infty\,.
\]
The limit as $s\to-\infty$ can be treated in a similar way.
\end{proof}
\par\bigskip
\noindent
\emph{Proof of Theorem~\ref{thm:res}.}
\par\noindent
Assume~$(b_-)$.
Let us show that $f$ satisfies the Cerami-Palais-Smale condition.
Let $(u_n)$ be a sequence in $W^{1,p}_0(\Omega)$ with
$f(u_n)$ bounded and $(1+\|u_n\|)\|f'(u_n)\| \to 0$.
First of all, let us show that $(u_n)$ is bounded in
$W^{1,p}_0(\Omega)$.
By contradiction, assume that $\|u_n\| \to \infty$
and set $z_n = \frac{u_n}{\|u_n\|}$.
Up to a subsequence, $z_n$ is convergent to some
$z$ weakly in $W^{1,p}_0(\Omega)$, strongly in $L^p(\Omega)$
and a.e. in~$\Omega$.
As in the proof of Theorem~\ref{thm:nonres},
we infer that $z_n \to z$ strongly in $W^{1,p}_0(\Omega)$
and $z\neq 0$.
\par
We also have
\[
\limsup_n \left|p f(u_n) - \langle f'(u_n),u_n\rangle \right|
<+\infty\,.
\]
Since
\[
p \Psi_{p,\kappa}(\xi) 
- \nabla\Psi_{p,\kappa}(\xi)\cdot\xi 
= \kappa^2 \left(\kappa^2+|\xi|^2\right)^{\frac{p-2}{2}}
- \kappa^p
\]
is bounded from below, we infer that
\[
\liminf_n 
\int_\Omega\left[pG(u_n)-g(u_n)u_n\right]\,dx > -\infty\,.
\]
On the other hand, there exists $C>0$ such that
\[
pG(s) - g(s)s \leq C \qquad\forall s\in\R
\]
whence, by Fatou's lemma,
\[
\int_\Omega \left\{\limsup_n 
\left[pG(u_n)-g(u_n)u_n\right]\right\}\,dx > -\infty\,.
\]
Since we have
\[
\lim_n\left[pG(u_n(x))-g(u_n(x))u_n(x)\right] = -\infty
\qquad\text{for a.e. $x\in\Omega$ with $z(x)\neq 0$}\,,
\]
we infer that $z=0$ a.e. in $\Omega$ and a contradiction follows.
\par
We conclude that the sequence $(u_n)$ is bounded, hence
convergent, up to a subsequence, to some
$u$ weakly in $W^{1,p}_0(\Omega)$.
As in the proof of Theorem~\ref{thm:nonres}, we get that 
$(u_n)$ is strongly convergent to $u$ in $W^{1,p}_0(\Omega)$.
\par
Now let 
\[
\lambda_{m_{\infty}} < \lambda \leq \lambda_{m_{\infty}+1}
\]
and define $X_-, X_+$ as in the proof of Theorem~\ref{thm:nonres}.
\par
We have
\[
\Psi_{p,\kappa}(\xi) \geq 
\frac{1}{p}\,|\xi|^p - \frac{1}{p}\,\kappa^p
\qquad\forall \xi\in\R^N
\]
and, by Proposition~\ref{prop:res},
\[
\lim_{|s|\to\infty}\, 
\left[G(s)-\frac{\lambda}{p}\,|s|^p\right] = -\infty\,.
\]
Therefore, there exists $C>0$ such that
\[
G(s) \leq 
\frac{\lambda}{p}\,|s|^p + C
\qquad\forall s\in\R\,.
\]
It easily follows that $\inf\limits_{X_+} f > -\infty$
and we conclude as in the proof of Theorem~\ref{thm:nonres}.
\par
Now assume~$(b_+)$, so that
\[
\lambda_{m_{\infty}} \leq \lambda < \lambda_{m_{\infty}+1}
\]
and either $1<p\leq 2$ with $\kappa\geq 0$ or 
$p>2$ with $\kappa=0$.
It follows that 
\[
p \Psi_{p,\kappa}(\xi) 
- \nabla\Psi_{p,\kappa}(\xi)\cdot\xi 
\]
is even bounded and the Cerami-Palais-Smale condition
can be proved as in the previous case.
\par
Now let us show that
\begin{equation}
\label{eq:lim-}
\lim_{\substack{\|u\|\to\infty \\ u\in X_-}}\,f(u) = -\infty\,.
\end{equation}
Let $u_n\in X_-$ with $\|u_n\|\to\infty$ and let
$z_n = \frac{u_n}{\|u_n\|}$.
Up to a subsequence, $(z_n)$ is convergent to some $z$
weakly in $W^{1,p}_0(\Omega)$, strongly in $L^p(\Omega)$
and a.e. in $\Omega$.
Since $z_n\in X_-$, we also have $z\neq 0$.
From Proposition~\ref{prop:res} we infer that
\[
\lim_{|s|\to\infty}\, 
\left[G(s)-\frac{\lambda}{p}\,|s|^p\right] = +\infty\,.
\]
In particular, there exists $C>0$ such that
\[
G(s) \geq 
\frac{\lambda}{p}\,|s|^p - C
\qquad\forall s\in\R\,.
\]
From Fatou's lemma we infer that
\[
\lim_n\, 
\int_\Omega\left[G(u_n)-\frac{\lambda}{p}\,|u_n|^p\right]
\,dx = +\infty\,.
\]
Since
\[
\Psi_{p,\kappa}(\xi) \leq \frac{1}{p}\,|\xi|^p 
\qquad\forall \xi\in\R^N\,,
\]
it follows that
\[
f(u_n) \leq 
\frac{1}{p}\, \int_{\Omega} \left[ |\nabla u_n|^p 
- {\lambda} |u_n|^p \right]\,dx 
- \int_{\Omega} \left[G(u_n) 
- \frac{\lambda}{p} |u_n|^p \right] \,dx\,,
\]
whence~\eqref{eq:lim-}.
Now we conclude as in the proof of Theorem~\ref{thm:nonres}.
\qed


%
\end{document}